\documentclass{amsart}[draft]
\usepackage{amsfonts, color}
\usepackage{latexsym}
\usepackage{amssymb}
\usepackage{amsmath}

\usepackage{amsthm}
\usepackage{epsfig}
\usepackage{graphicx,colortbl}
\usepackage{url}
\usepackage[colorlinks=false, pdfborder={0 0 0}]{hyperref}
\usepackage{mathtools, color, changes}


\newcommand{\R}{\mathbb R}
\newcommand{\N}{\mathbb N}
\newcommand{\Z}{\mathbb Z}

\newcommand{\E}{\mathbb E}
\newcommand{\Pro}{\mathbb P}

\newcommand\subs[2]{\genfrac{}{}{0pt}{}{#1}{#2}}

\newtheorem{thm}{Theorem}[section]
\newtheorem{cor}{Corollary}[section]
\newtheorem{lemma}{Lemma}[section]

\theoremstyle{remark}
\newtheorem{rmk}{Remark}

\numberwithin{equation}{section}

\begin{document}


\title[Discrete Borell's inequality and mean width of random polytopes]{Borell's inequality and mean width of random polytopes via discrete inequalities}

\author[D.\,Alonso]{David Alonso-Guti\'errez}
\address{\'Area de an\'alisis matem\'atico, Departamento de matem\'aticas, Facultad de Ciencias, Universidad de Zaragoza, Pedro Cerbuna 12, 50009 Zaragoza (Spain), IUMA}
\email[(David Alonso)]{alonsod@unizar.es}

\author[L.C. Garc\'ia-Lirola]{Luis C. Garc\'ia-Lirola}
\address{\'Area de an\'alisis matem\'atico, Departamento de matem\'aticas, Facultad de Ciencias, Universidad de Zaragoza, Pedro Cerbuna 12, 50009 Zaragoza (Spain), IUMA}
\email[(Luis C. Garc\'ia-Lirola)]{luiscarlos@unizar.es}

\thanks{The authors are partially supported by Project PID2022-137294NB-I00  funded by   MICIU/AEI/10.13039/501100011033/ and FEDER Una manera de hacer Europa, and by DGA Project E48\_23R. The second named author is also supported by Project PID2021-122126NB-C32 funded by   MICIU/AEI/10.13039/501100011033/ and FEDER Una manera de hacer Europa, Project 219551/PI/22 funded by Fundación Séneca - ACyT Región de Murcia, and by Generalitat Valenciana Grant CIGE/2022/97. 
}
\begin{abstract}
Borell's inequality states the existence of a positive absolute constant $C>0$ such that for every $1\leq p\leq q$
$$
\left(\E|\langle X, e_n\rangle|^p\right)^\frac{1}{p}\leq\left(\E|\langle X, e_n\rangle|^q\right)^\frac{1}{q}\leq C\frac{q}{p}\left(\E|\langle X, e_n\rangle|^p\right)^\frac{1}{p},
$$
whenever $X$ is a random vector uniformly distributed on any convex body $K\subseteq\R^n$ and $(e_i)_{i=1}^n$ is the standard canonical basis in $\R^n$. In this paper, we will prove a discrete version of this inequality, which will hold whenever $X$ is a random vector uniformly distributed on $K\cap\Z^n$ for any convex body $K\subseteq\R^n$ containing the origin in its interior. We will also make use of such discrete version to obtain discrete inequalities from which we can recover the estimate $\E w(K_N)\sim w(Z_{\log N}(K))$ for any convex body $K$ containing the origin in its interior, where $K_N$ is the centrally symmetric random polytope $K_N=\textrm{conv}\{\pm X_1,\ldots,\pm X_N\}$ generated by independent random vectors uniformly distributed on $K$, $Z_{p}(K)$ is the $L_p$-centroid body of $K$ for any $p\geq1$, and $w(\cdot)$ denotes the mean width.
\end{abstract}

\date{\today}
\maketitle

\section{Introduction and Notation}

Given a convex body $K\subseteq\R^n$ and $p\geq1$, its $L_p$-centroid body, $Z_p(K)$, is defined as the convex body with support function
$$
h_{Z_p(K)}(y)=\left(\E|\langle X,y\rangle|^p\right)^\frac{1}{p}=\left(\frac{1}{|K|}\int_K|\langle x,y\rangle|^pdx\right)^\frac{1}{p}\quad\forall y\in\R^n,
$$
where $X$ is a random vector uniformly distributed on $K$, $\E Y$ denotes the expectation of a random variable $Y$, and $|K|$ denotes the volume of $K$ (i.e., its $n$-dimensional Lebesgue measure). Let us recall that, given any convex body $L\subseteq\R^n$, its support function $h_L$ is defined as $\displaystyle{h_L(y):=\sup_{x\in L}\langle x,y\rangle}$ for every $y\in\R^n$ and that every sublinear function $f:\R^n\to\R$ is the support function of a unique convex body (see \cite[Thm 1.7.1]{Sch}). These bodies were introduced, under a different normalization, in \cite{LYZ}. Since then, they have shown to play a crucial role on the distribution of the mass on convex bodies (see e.g. \cite{FGP, GSTV, P}).

By H\"older's inequality, for any convex body $K\subseteq\R^n$ and any $1\leq p\leq q$, if $X$ is a random vector uniformly distributed on $K$ we have that for any $\theta\in S^{n-1}$, the Euclidean sphere in $\R^n$,
\begin{equation}\label{eq:HolderDirection}
h_{Z_p(K)}(\theta)=\left(\E|\langle X,\theta\rangle|^p\right)^\frac{1}{p}\leq \left(\E|\langle X,\theta\rangle|^q\right)^\frac{1}{q}=h_{Z_q(K)}(\theta).
\end{equation}
Equivalently,
\begin{equation}\label{eq:Holder}
Z_p(K)\subseteq Z_q(K).
\end{equation}

As a consequence of Borell's inequality (see \cite[Appendix III]{MS} and \cite[Thm. 2.4.6 and Rmk. 2.4.8]{BGVV}), we have that there exists an absolute constant $C>0$ such that for any convex body $K\subseteq\R^n$ and any $\theta\in S^{n-1}$, if $X$ is a random vector uniformly distributed on $K$, then for any $1\leq p\leq q$,
\begin{equation}\label{eq:ContinuousBorell}
\left(\E|\langle X,\theta\rangle|^q\right)^\frac{1}{q}\leq C\frac{q}{p}\left(\E|\langle X,\theta\rangle|^p\right)^\frac{1}{p}.
\end{equation}
Equivalently, there exists an absolute constant $C>0$ such that for any convex body $K\subseteq\R^n$, and for any $1\leq p\leq q$,
\begin{equation}\label{eq:BorellCentroidBodies}
Z_q(K)\subseteq C\frac{q}{p}Z_p(K).
\end{equation}
Writing equations \eqref{eq:Holder} and \eqref{eq:BorellCentroidBodies} together, we have that there exists an absolute constant $C>0$ such that for any convex body $K\subseteq\R^n$ and any $1\leq p\leq q$,
\begin{equation}\label{eq:HolderAndBorell}
Z_p(K)\subseteq Z_q(K)\subseteq C\frac{q}{p} Z_p(K).
\end{equation}
Let us stress out that the constant $C$ in the equations \eqref{eq:ContinuousBorell} to \eqref{eq:HolderAndBorell} is a positive absolute constant, independent of the dimension $n$, of the convex body $K\subseteq\R^n$, and of the parameters $1\leq p\leq q$. Let us also point out that the fact that \eqref{eq:HolderDirection} and \eqref{eq:ContinuousBorell} hold for every convex body $K$ and every $\theta\in S^{n-1}$ is equivalent to the fact that they hold for every convex body $K\subseteq\R^n$ and $\theta=e_n$, where  $(e_i)_{i=1}^n$ is the canonical basis in $\R^n$. That is, for every convex body $K\subseteq\R^n$ and $1\leq p\leq q$, if $X$ is a random vector uniformly distributed on $K$ then
\begin{equation}\label{eq:ContinuousBorellOneDirection}
\left(\E|\langle X,e_n\rangle|^p\right)^\frac{1}{p}\leq\left(\E|\langle X,e_n\rangle|^q\right)^\frac{1}{q}\leq C\frac{q}{p}\left(\E|\langle X,e_n\rangle|^p\right)^\frac{1}{p}.
\end{equation}

Indeed, given  $K\subseteq\R^n$ a convex body and $\theta\in S^{n-1}$, there exists an orthogonal map $U\in O(n)$ such that $U^t e_n=\theta$, where $U^t$ denotes the transpose matrix of $U$. Thus, if $X$ is a random vector uniformly distributed on $K$, then for any $p\geq1$
$$
\left(\E|\langle X,\theta\rangle|^p\right)^\frac{1}{p}=\left(\E|\langle X,U^te_n\rangle|^p\right)^\frac{1}{p}=\left(\E|\langle UX,e_n\rangle|^p\right)^\frac{1}{p}=\left(\E|\langle Y,e_n\rangle|^p\right)^\frac{1}{p},
$$
where $Y$ is a random vector uniformly distributed on $UK$. Therefore, applying \eqref{eq:ContinuousBorellOneDirection} to the convex body $UK$, we obtain \eqref{eq:HolderDirection} and \eqref{eq:ContinuousBorell}.

The main purpose of this paper is to obtain a discrete version of \eqref{eq:ContinuousBorellOneDirection}, in which the random vector $X$ is uniformly distributed on the intersection of a convex body $K\subseteq\R^n$ with the integer lattice $\Z^n$, following an active line of research whose purpose is to obtain discrete analogues of classical results in convex geometry (see, for instance, the discrete versions of the Brunn-Minkowski inequality proved by Gardner and Zhang in \cite{GG}, Koldobsky's slicing inequality \cite{AHZ}, or Meyer's inequality \cite{FH}).

In \cite{IYNZ}, the authors obtained a discrete Brunn-Minkowski inequality for the lattice point enumerator $G_n(\cdot)$, which is defined as the cardinality of the intersection of a set in $\R^n$ with $\Z^n$ (see Theorem \ref{thm: BM_lattice_point_no_G(K)G(L)>0} below). Such version of the Brunn-Minkowski inequality allows to recover the classical one (see \cite[Theorem 7.1.1]{Sch}), and has been very useful to obtain some other discrete versions of inequalities in convex geometry (see, for instance, \cite{ILY} for a discrete version of an isoperimetric inequality or \cite{ALY} for discrete versions of Rogers-Shephard type inequalities and related inequalities). Due to the fact that adding an open unit cube in the large term in this discrete version of the Brunn-Minkowski inequality is necessary, typically an extra open unit cube appears in some terms in these new discrete inequalities. Nevertheless, one can still deduce their continuous versions from them.

In this paper, we will mainly consider convex bodies $K\subseteq\R^n$ with $0\in\operatorname{int}K$, the interior of $K$, and $\displaystyle{\max_{x\in K\cap\Z^n}|\langle x, e_n\rangle|\geq1}$. Such condition will ensure that $K\cap\Z^n\neq\emptyset$ (i.e., $G_n(K)\neq0)$ and that if $X$ is a random vector uniformly distributed on $K\cap\Z^n$, then $\E|\langle X, e_n\rangle|^p\neq 0$ for every $p\geq1$. Besides, for any convex body $K\subseteq\R^n$ with $0\in\operatorname{int}K$,  we have that $\lambda UK$ will satisfy such condition for any orthogonal transformation $U$, if $\lambda>0$ is large enough, which will allow to recover continuous inequalities from discrete ones for convex bodies containing the origin in its interior.

Unfortunately, given $1\leq p\leq q$ and $n\in \mathbb N$, it is not possible to find $C(p,q,n)>0$ such that $\left(\E|\langle X, e_n\rangle|^q\right)^\frac{1}{q}\leq C(p,q,n)\left(\E|\langle X, e_n\rangle|^p\right)^\frac{1}{p}$ for every convex body $K\subset \mathbb R^n$ with $0\in \operatorname{int} K$ and $\displaystyle{\max_{x\in K\cap\Z^n}|\langle x, e_n\rangle|\geq1}$. This can be easily checked by considering $K= \operatorname{conv}(\{(x,-1/2)\in \mathbb R^{n-1}\times \mathbb R
	: \|x\|_\infty\leq \lambda\}\cup\{e_n\})$ and letting $\lambda\to\infty$, where $\operatorname{conv}(A)$ denotes the convex hull of the set $A\subseteq\R^n$. However, as a direct consequence of the fact that, if $1\leq p\leq q$, then the $\ell^q$-norm is bounded by the $\ell^p$-norm in $\R^N$ for any $N\in\N$, we have that for every convex body $K\subseteq\R^n$ with $G_n(K)\neq0$, then
$$
\left(\E|\langle X, e_n\rangle|^q\right)^\frac{1}{q}\leq G_n(K)^{\frac{1}{p}-\frac{1}{q}}(\E|\langle X,e_n\rangle|^p)^\frac{1}{p}.
$$
Therefore, if we fix $1\leq p\leq q$ and consider convex bodies such that $G_n(K)$ is bounded above by $C\frac{q}{p}$ with $C$ an absolute constant, we have that $G_n(K)^{\frac{1}{p}-\frac{1}{q}}\leq G_n(K)\leq C\frac{q}{p}$, and for such family of convex bodies we trivially obtain the following discrete version of Borell's inequality:
\begin{equation}\label{eq:DiscreteBorellSmallK}
\left(\E|\langle X, e_n\rangle|^q\right)^\frac{1}{q}\leq C\frac{q}{p}(\E|\langle X,e_n\rangle|^p)^\frac{1}{p}.
\end{equation}

The following theorem, and its equivalent form given by Theorem \ref{thm:DiscreteBorell} below, gives a discrete version of Borell's inequality which also holds whenever $G_n(K)$ is not bounded. As we will see, when applied to $\lambda K$ for a fixed convex body $K\subseteq\R^n$ with the origin in its interior and taking $\lambda\to\infty$, we will recover the continuous version of Borell's inequality. Notice that in such case $G_n(\lambda K)$ will tend to infinity and the estimate given by \eqref{eq:DiscreteBorellSmallK} cannot be applied.

\begin{thm}\label{thm:DiscreteBorellNew}

There exists an absolute constant $C>0$ such that for any convex body $K\subseteq\R^n$ with $G_n(K)\neq 0$ and any $1\leq p\leq q$
\begin{equation}\label{eq:DiscreteBorellNew}
	\left(\E|\langle X, e_n\rangle|^q\right)^\frac{1}{q}\leq C\frac{q}{p}\left(1+(\E|\langle Y,e_n\rangle|^p)^\frac{1}{p}\left(\frac{G_n(K+(-1,1)^n)}{G_n(K)}\right)^\frac{1}{p}\right),
\end{equation}
where $X$ is a random vector uniformly distributed on $K\cap \mathbb Z^n$, and $Y$ is a random vector uniformly distributed on $(K+(-1,1)^n)\cap \Z^n$.

\end{thm}

Let us introduce the following notation, which will allow us to express the right-hand side of Theorem \ref{thm:DiscreteBorellNew} in terms of the random vector $X$ uniformly distributed on $K\cap\Z^n$ and, studying the asymptotic behavior of the quantity introduced here, prove that it implies \eqref{eq:ContinuousBorell} whenever $0\in\operatorname{int}K$: For any convex body $K\subseteq\R^n$ such that $\displaystyle{\max_{x\in K\cap\Z^n}|\langle x, e_n\rangle|\geq1}$ and any $p\geq 1$, we define $D_0(K,p)$ as the number given by
$$
D_0(K,p):=\frac{\left(1+(\E|\langle Y,e_n\rangle|^p)^\frac{1}{p}\left(\frac{G_n(K+C_n)}{G_n(K)}\right)^\frac{1}{p}\right)}{\left(\E|\langle X, e_n\rangle|^p\right)^\frac{1}{p}},
$$
where $C_n=(-1,1)^n$, $G_n(K)$ is the cardinality of $K\cap \mathbb Z^n$, $X$ is a random vector uniformly distributed on $K\cap\Z^n$, and $Y$ is a random vector uniformly distributed on $(K+C_n)\cap \Z^n$. With this notation, and taking into account H\"older's inequality, we can rewrite Theorem \ref{thm:DiscreteBorellNew} as follows in the case that $\displaystyle{\max_{x\in K\cap\Z^n}|\langle x, e_n\rangle|\geq1}$ (otherwise, Theorem \ref{thm:DiscreteBorellNew} is trivial):

\begin{thm}\label{thm:DiscreteBorell}
There exists an absolute constant $C>0$ such that for any convex body $K\subseteq\R^n$ with  $\displaystyle{\max_{x\in K\cap\Z^n}|\langle x, e_n\rangle|\geq1}$, and any $1\leq p\leq q$
\begin{equation}\label{eq:DiscreteBorell}
\left(\E|\langle X, e_n\rangle|^p\right)^\frac{1}{p}\leq\left(\E|\langle X, e_n\rangle|^q\right)^\frac{1}{q}\leq C\frac{q}{p}D_0(K,p)\left(\E|\langle X, e_n\rangle|^p\right)^\frac{1}{p},
\end{equation}
where $X$ is a random vector uniformly distributed on $K\cap \mathbb Z^n$.
\end{thm}

In order to study the asymptotic behavior of the the number $D_0(\lambda K, p)$ as $\lambda\to\infty$ and recover continuous inequalities from their discrete versions, we also define, for any convex body $K\subseteq\R^n$ such that the Euclidean closed unit ball, $B_2^n$, satisfies $B_2^n\subseteq K$, and any $q\geq1$, $D_{q}(K)$ as the following number
$$
D_{q}(K):=\sup_{\subs{U\in O(n)}{1\leq p\leq q}}D_0(UK,p).
$$
Let us point out that the condition $B_2^n\subseteq K$ in this definition is imposed so that, for every $U\in O(n)$, $e_n\in UK\cap\Z^n$  and then for any $p\geq 1$ and any $U\in O(n)$, $D_0(UK,p)$ is well defined.

The following lemma, involving the role of the number $D_q(\cdot)$, will allow us to recover \eqref{eq:ContinuousBorellOneDirection} from Theorem \ref{thm:DiscreteBorell} (with the same value of the constant $C$), as well as some other continuous inequalities from discrete ones:
\begin{lemma}\label{lem:LimitC}
For any convex body $K\subseteq\R^n$ such that $0\in\operatorname{int}K$ and any $q\geq 1$ we have
$$
\lim_{\lambda\to\infty}D_{q}(\lambda K)=1.
$$
\end{lemma}

\begin{rmk}\label{rmk:LowerboundC0}
Let us point out that for every convex body $K\subseteq\R^n$ such that $\displaystyle{\max_{x\in K\cap\Z^n}|\langle x, e_n\rangle|\geq1}$ and any $p\geq 1$ we have that $D_0(K,p)\geq 1$. Indeed, if $X$ and $Y$ are random vectors uniformly distributed on $K\cap\Z^n$ and $(K+C_n)\cap\Z^n$ respectively,
\begin{eqnarray*}
D_0(K,p)
&\geq&\frac{\left((\E|\langle Y,e_n\rangle|^p)^\frac{1}{p}\left(\frac{G_n(K+C_n)}{G_n(K)}\right)^\frac{1}{p}\right)}{\left(\E|\langle X, e_n\rangle|^p\right)^\frac{1}{p}}\cr
&=&\left(\frac{\sum_{x\in (K+C_n)\cap\Z^n}|\langle x, e_n\rangle|^p}{\sum_{x\in K\cap\Z^n}|\langle x, e_n\rangle|^p}\right)^\frac{1}{p}\geq\left(\frac{\sum_{x\in K\cap\Z^n}|\langle x, e_n\rangle|^p}{\sum_{x\in K\cap\Z^n}|\langle x, e_n\rangle|^p}\right)^\frac{1}{p}=1.
\end{eqnarray*}
For every convex body $K\subseteq\R^n$ such that $0\in\operatorname{int}K$, there exists $\lambda_0>0$ such that $B_2^n\subseteq \lambda_0 K$. Then, for every $\lambda\geq\lambda_0$, $U\in O(n)$, and $p\geq 1$, choosing $q\geq p$, we have that
$$
1\leq D_0(\lambda UK, p)\leq\sup_{\subs{U\in O(n)}{1\leq p\leq q}}D_0(\lambda UK,p)=D_q(\lambda K).
$$
Therefore, Lemma \ref{lem:LimitC} implies that for every $U\in O(n)$ and $p\geq 1$ we have $\displaystyle{\lim_{\lambda\to\infty}D_0(\lambda UK, p)=1}$.
\end{rmk}

Random polytopes are important objects of study in several areas in mathematics since Sylvester initiated their study with a problem posed in \textit{The Educational Times} in 1864 \cite{S}. This problem was originally ill-posed and the question was modified within a year, and became known as the ``four-point problem''. It were R\'enyi and Sulanke who, in their seminal papers \cite{RS}, \cite{RS2}, \cite{RS3}, focused on the expected volume of random polytopes as the number of points generating it tends to infinity, as well as on the vertex number in the planar case. Random polytopes appear in approximation theory \cite{B}, random matrix theory \cite{LPRT}, or in other disciplines such as statistics, information theory, signal processing, medical imaging or digital communications (see \cite{DT} and the references therein). Besides, since the work of Gluskin \cite{G}, they have been known to provide counterexamples to some conjectures, or examples to the sharpness of some estimates on convex bodies (see the survey \cite{MT} and its references, or \cite{GL} and \cite{LRT}), as the behavior of some parameters of a linear and geometric nature different from the known deterministic constructions.

Different quantities such as expectations, variances, distribution, concentration inequalities, central limit theorems or large deviation principles of several geometric functionals, such as volume, number of faces, or querma\ss integrals, associated to random polytopes are studied, which obviously depend on the random model from which the random polytope is generated (see, for instance, \cite{A}, \cite{AS}, \cite{BV}, \cite{DGT}, \cite{GT}, or \cite{HR}). Here we will focus on the mean width and the following model: Given a convex body $K\subseteq\R^n$ with $0\in\operatorname{int}K$, $N\geq n$, and $X_1,\ldots, X_N$ independent random vectors uniformly distributed on $K$, let $K_N=\operatorname{conv}\{\pm X_1,\dots, \pm X_N\}$ be the centrally symmetric random polytope generated by $X_1,\dots, X_N$. It is well known that there exist absolute constants $c_1, c_1^\prime,c_2, c_2^\prime$ such that for any convex body $K\subseteq\R^n$  with $0\in\operatorname{int}K$, and any $n\leq N\leq e^n$ we have
\begin{equation}\label{eq:MeanWidthContinuous}
c_1w(Z_{c_1^\prime\log N}(K))\leq \E w(K_N)\leq c_2w(Z_{c_2^\prime\log N}(K)),
\end{equation}
where for any convex body $L\subseteq\R^n$ containing the origin, $w(L)$ denotes its mean width
$$
w(L)=\int_{S^{n-1}}h_L(\theta)d\sigma(\theta),
$$
being $d\sigma$ the uniform probability measure on $S^{n-1}$. We refer the reader to \cite[Proposition 11.3.10]{BGVV} for a proof of the upper bound and to \cite[Theorem 11.3.2]{BGVV} for a proof of a stronger result that implies the lower bound when $N\geq n^2$, which relies on the continuous Borell's inequality given by \eqref{eq:ContinuousBorell}. The fact that the lower bound also holds when $n\leq N\leq n^2$ is probably folklore. Nevertheless, since we could not find it in the literature, we will include a proof, which relies on Borell's inequality as well, in Section \ref{subsec:ExpectedMeanWidthRandomPOlytopes}. We also refer the reader to \cite{AP} for the study of the expected value of the mean width of a large family of random convex sets, which include random polytopes generated by isotropic log-concave random vectors.

In view of the connection of Borell's inequality to the estimate of the mean width of random polytopes given by \eqref{eq:MeanWidthContinuous}, and the relation between Theorem \ref{thm:DiscreteBorell} and the continuous version of Borell's inequality, it is our purpose to explore the relation between Theorem \ref{thm:DiscreteBorell}, and discrete versions of \eqref{eq:MeanWidthContinuous}. However, since the inequality given by Theorem \ref{thm:DiscreteBorell} is only provided in the direction given by $\theta=e_n$, in order to explore such relation, we must rewrite \eqref{eq:MeanWidthContinuous} in a different way.

On the one hand, notice that, by uniqueness of the Haar measure, we have that if $d\nu$ denotes the Haar probability measure on the set of orthogonal matrices $O(n)$, then
\begin{eqnarray*}
w(K_N)&=&\int_{S^{n-1}}h_{K_N}(\theta)d\sigma(\theta)=\int_{S^{n-1}}\max_{1\leq i\leq N}|\langle X_i,\theta\rangle|d\sigma(\theta)\cr
&=&\int_{O(n)}\max_{1\leq i\leq N}|\langle X_i,U^te_n\rangle|d\nu(U)=\int_{O(n)}\max_{1\leq i\leq N}|\langle UX_i,e_n\rangle|d\nu(U).
\end{eqnarray*}
Therefore,
\begin{eqnarray*}
\E w(K_N)&=&\E\int_{O(n)}\max_{1\leq i\leq N}|\langle UX_i,e_n\rangle|d\nu(U)=\int_{O(n)}\E\max_{1\leq i\leq N}|\langle UX_i,e_n\rangle|d\nu(U)\cr
&=&\int_{O(n)}\frac{1}{|K|^N}\int_{K}\dots\int_K\max_{1\leq i\leq N}|\langle Ux_i,e_n\rangle|dx_N\dots dx_1d\nu(U)\cr
&=&\int_{O(n)}\frac{1}{|UK|^N}\int_{UK}\dots\int_{UK}\max_{1\leq i\leq N}|\langle x_i,e_n\rangle|dx_N\dots dx_1d\nu(U).\cr
\end{eqnarray*}
On the other hand, for any $p\geq1$ we have that
\begin{eqnarray*}
w(Z_p(K))&=&\int_{S^{n-1}}h_{Z_p(K)}(\theta)d\sigma(\theta)=\int_{S^{n-1}}\left(\frac{1}{|K|}\int_K|\langle x,\theta\rangle|^pdx\right)^\frac{1}{p}d\sigma(\theta)\cr
&=&\int_{O(n)}\left(\frac{1}{|K|}\int_K|\langle x,U^te_n\rangle|^pdx\right)^\frac{1}{p}d\nu(U)\cr
&=&\int_{O(n)}\left(\frac{1}{|K|}\int_K|\langle Ux,e_n\rangle|^pdx\right)^\frac{1}{p}d\nu(U)\cr
&=&\int_{O(n)}\left(\frac{1}{|UK|}\int_{UK}|\langle x,e_n\rangle|^pdx\right)^\frac{1}{p}d\nu(U).\cr
\end{eqnarray*}
Therefore, Equation \eqref{eq:MeanWidthContinuous} is equivalent to
\begin{align}\label{eq:MeanWidthRotations}
&c_1\int_{O(n)}\left(\frac{1}{|UK|}\int_{UK}|\langle x,e_n\rangle|^{c_1^\prime\log N}dx\right)^\frac{1}{c_1^\prime\log N}d\nu(U)\cr
&\leq\int_{O(n)}\frac{1}{|UK|^N}\int_{UK}\dots\int_{UK}\max_{1\leq i\leq N}|\langle x_i,e_n\rangle|dx_N\dots dx_1d\nu(U)\cr
&\leq c_2\int_{O(n)}\left(\frac{1}{|UK|}\int_{UK}|\langle x,e_n\rangle|^{c_2^\prime\log N}dx\right)^\frac{1}{c_2^\prime\log N}d\nu(U).
\end{align}
In this paper we will prove the following two theorems, which will provide discrete versions of the inequalities in \eqref{eq:MeanWidthRotations} where the role of the volume is played by the lattice point enumerator measure given by $G_n(L)$, the cardinality of the set $L\cap\Z^n$ for any Borel set $L\subseteq\R^n$. The two inequalities in \eqref{eq:MeanWidthRotations} and, equivalently, equation \eqref{eq:MeanWidthContinuous} will be recovered from the discrete versions as corollaries. Let us also point out that, as a consequence of \eqref{eq:HolderAndBorell}, inequality \eqref{eq:MeanWidthContinuous} can be equivalently written with constants $c_1^\prime=c_2^\prime=1$. Nevertheless, for our purposes of finding analogue discrete inequalities we decided not to do so.

\begin{thm}\label{thm:DiscreteUpperBoundMeanWidth}
There exists an absolute constant $C>0$ such that for every convex body $K\subseteq\R^n$ with $0\in K$ and every $N\in\N$ with $N\geq 3$,
\begin{eqnarray*}
&&\int_{O(n)}\frac{1}{(G_n(UK))^N}\sum_{x_1,\ldots,x_N\in UK\cap\Z^n}\max_{1\leq i\leq N}|\langle x_i,e_n\rangle|d\nu(U)\cr
&\leq& C\int_{O(n)}\left(\frac{1}{G_n(UK)}\sum_{x\in UK\cap\Z^n}|\langle x,e_n\rangle|^{\log N}\right)^\frac{1}{\log N}d\nu(U).
\end{eqnarray*}
\end{thm}

As a consequence, the right-hand side inequality in \eqref{eq:MeanWidthContinuous} will be recovered.

\begin{thm}\label{thm:DiscreteLowerBoundMeanWidth}
There exist absolute constants $c,C>0$ such that if  $K\subseteq\R^n$ is a convex body with $B_2^n\subseteq K$, $q\geq 1$ and $N\in\N$ with $(CD_{q}(K))^2\leq N\leq e^{q}$, then
\begin{eqnarray*}
&&\frac{1}{(G_n(K))^N}\sum_{x_1,\ldots,x_N\in K\cap\Z^n}\max_{1\leq i\leq N}|\langle x_i,e_n\rangle|\cr
&\geq& c\left(\frac{1}{G_n(K)}\sum_{x\in K\cap\Z^n}|\langle x,e_n\rangle|^\frac{\log N}{2\log(CD_{q}(K))}\right)^\frac{2\log(CD_{q}(K))}{\log N}.
\end{eqnarray*}
\end{thm}

Given a convex body $K\subseteq\R^n$ with $B_2^n\subseteq K$, and $U\in O(n)$ the convex body $UK$ satisfies that $B_2^n\subseteq UK$ and $D_q(UK)=D_q(K)$. Applying Theorem \ref{thm:DiscreteLowerBoundMeanWidth} to $UK$ and integrating in $U\in O(n)$, we obtain the following corollary, which provides a discrete version of the right hand-side inequality in \eqref{eq:MeanWidthRotations}:

\begin{cor}\label{cor:DiscreteLowerBoundMeanWidth}
There exist absolute constants $c,C>0$ such that if  $K\subseteq\R^n$ is a convex body with $B_2^n\subseteq K$, $q\geq 1$ and $N\in\N$ with $(CD_{q}(K))^2\leq N\leq e^{q}$, then
\begin{eqnarray*}
&&\int_{O(n)}\frac{1}{(G_n(UK))^N}\sum_{x_1,\ldots,x_N\in UK\cap\Z^n}\max_{1\leq i\leq N}|\langle x_i,e_n\rangle|d\nu(U)\cr
&\geq& c\int_{O(n)}\left(\frac{1}{G_n(UK)}\sum_{x\in UK\cap\Z^n}|\langle x,e_n\rangle|^\frac{\log N}{2\log(CD_{q}(K))}\right)^\frac{2\log(CD_{q}(K))}{\log N}d\nu(U).
\end{eqnarray*}
\end{cor}

As a consequence, the left-hand side  inequality in \eqref{eq:MeanWidthContinuous} will be recovered.

\begin{rmk}\label{rmk:Nandq}
Let us point out that, given a fixed $N\in\N$, it is necessary that $q\geq\log N$ in order to fulfill the conditions in Theorem \ref{thm:DiscreteLowerBoundMeanWidth} and Corollary \ref{cor:DiscreteLowerBoundMeanWidth}. However, it might not be possible to find a number $q\geq 1$ such that $(CD_q(K))^2\leq N$, as the following example in $\R^2$ shows: let, for any $M\in\N$, $K_M=\textrm{conv}\left(B_2^2\cup\{(M,0),(-M,0)\}\right)\subseteq\R^2$. This convex body satisfies that $B_2^2\subseteq K_M$ and for any $p\geq 1$
\begin{eqnarray*}
D_0(K_M,p)&=&\frac{1+(\E|\langle Y,e_2\rangle|^p)^\frac{1}{p}\left(\frac{G_2(K_M+C_2)}{G_2(K_M)}\right)^\frac{1}{p}}{\left(\E|\langle X, e_2\rangle|^p\right)^\frac{1}{p}}\cr
&=&\frac{G_2(K_M)^\frac{1}{p}+\left(\sum_{x\in(K_M+C_2)\cap\Z^2}|\langle x,e_2\rangle|^p\right)^\frac{1}{p}}{\left(\sum_{x\in K_M\cap\Z^2}|\langle x,e_2\rangle|^p\right)^\frac{1}{p}}\cr
&=&\frac{(2M+3)^\frac{1}{p}+(2(2M+1))^\frac{1}{p}}{2^\frac{1}{p}}=\left(M+\frac{3}{2}\right)^\frac{1}{p}+(2M+1)^\frac{1}{p}.
\end{eqnarray*}
Thus, for any $q\geq1$,
$$
D_q(K_M)\geq\sup_{1\leq p\leq q}\left(\left(M+\frac{3}{2}\right)^\frac{1}{p}+(2M+1)^\frac{1}{p}\right)=M+\frac{3}{2}+2M+1=3M+\frac{5}{2}
$$
and, given a fixed $N\in\N$, taking the convex body $K_M$ with $M$ large enough there exists no $q\geq1$ for which the condition $(CD_q(K_M))^2\leq N$ is fulfilled.

However, given a fixed convex body $K\subseteq\R^n$ with $B_2^n\subseteq K$, H\"older's inequality implies that
\begin{eqnarray*}
D_q(K)&=&\sup_{\subs{U\in O(n)}{1\leq p\leq q}}\frac{1+(\E|\langle Y_U,e_n\rangle|^p)^\frac{1}{p}\left(\frac{G_n(UK+C_n)}{G_n(UK)}\right)^\frac{1}{p}}{\left(\E|\langle X_U, e_n\rangle|^p\right)^\frac{1}{p}}\cr
&\leq&\sup_{U\in O(n)}\frac{1+h_{UK}(e_n)\frac{G_n(UK+C_n)}{G_n(UK)}}{\E|\langle X_U, e_n\rangle|}\cr
&=&\sup_{U\in O(n)}\frac{2+h_{UK}(e_n)\frac{G_n(UK+C_n)}{G_n(UK)}}{\E|\langle X_U, e_n\rangle|}\cr
\end{eqnarray*}
where, for every $U\in O(n)$, $X_U$ denotes a random vector uniformly distributed on $UK\cap\Z^n$ and $Y_U$ denotes a random vector uniformly distributed on $(UK+C_n)\cap\Z^n$. Therefore, given a fixed convex body $K\subseteq\R^n$ with $B_2^n\subseteq K$ and calling
$$
D(K):=\sup_{U\in O(n)}\frac{2+h_{UK}(e_n)\frac{G_n(UK+C_n)}{G_n(UK)}}{\E|\langle X_U, e_n\rangle|},
$$
we have that for every $N\geq (CD(K))^2$, taking $q=\log N$, the conditions in Theorem \ref{thm:DiscreteLowerBoundMeanWidth} are fulfilled.

Moreover, given a fixed convex body $K\subseteq\R^n$ with $B_2^n\subseteq K$ and a fixed  $N\in\N$, with $N>C^2$, taking $q=\log N$, Lemma \ref{lem:LimitC} implies that if $\lambda>0$ is large enough, then the condition $(CD_q(\lambda K))^2\leq N\leq e^q$ will be fulfilled.
\end{rmk}
\section{Preliminaries}\label{sec:Preliminaries}

In this section we will introduce some well-known results that will be used in our proofs.

\subsection{The lattice point enumerator}\label{subsec:LatticePoint enumerator}
Let us recall that the lattice point enumerator measure, $dG_n$, is the measure on $\R^n$ given, for any Borel set $L\subseteq\R^n$, by
$$
G_n(L)=\sharp(L\cap\Z^n),
$$
where $\sharp(\cdot)$ denotes cardinality of a set. 

The measure given by the lattice point enumerator satisfies the following discrete version of the Brunn-Minkowski inequality, which was proved in \cite[Theorem 2.1]{IYNZ}, and from which one can recover the classical one \cite[Theorem 7.1.1]{Sch}. It reads as
follows:
\begin{thm}\label{thm: BM_lattice_point_no_G(K)G(L)>0}
Let $\lambda\in(0,1)$ and let $K,L\subset\R^n$ be non-empty bounded Borel sets. Then
$$
G_n\left((1-\lambda)K+\lambda L+(-1,1)^n\right)^\frac{1}{n}\geq(1-\lambda)G_n(K)^\frac{1}{n}+\lambda G_n(L)^\frac{1}{n}.
$$
\end{thm}

The measure  $dG_n$ also satisfies (see \cite[Lemma 3.22]{TV} and \cite[Section 3.1]{ALY}) that for any convex body $K\subseteq\R^n$ and any bounded set $M$ containing the origin
\begin{equation}\label{eq:ApproachingVolByDilationsWithG}
\lim_{\lambda\to\infty}\frac{G_n(\lambda K+M)}{\lambda^n}=|K|.
\end{equation}
Moreover, for any $f:K\to\R$ which is Riemann-integrable on $K$, we have that
\begin{equation}\label{eq:ApproachingIntegralWithG}
\lim_{\lambda\to\infty}\frac{1}{\lambda^n}\int_{ \lambda K}f\left(\frac{x}{\lambda}\right)dG_n(x)=\lim_{\lambda\to\infty}\frac{1}{\lambda^n}\sum_{x\in K\cap\left(\frac{1}{\lambda}\Z^n\right)}f(x)=\int_{K}f(x)dx,
\end{equation}
where the first identity follows from the definition of the measure $dG_n$. The second identity can be obtained by extending the function $f$ to a rectangle containing $K$ as $f(x)=0$ for every $x\not\in K$, which is Riemann-integrable on the rectangle, and applying \cite[Proposition 6.3]{CS}, which is valid for Riemann-integrable functions on the rectangle. Notice that $\displaystyle{\sum_{x\in K\cap\left(\frac{1}{\lambda}\Z^n\right)}f(x)}$ is a Riemann sum of the extension of $f$ to the rectangle, corresponding to the partition given by the rectangle intersected with cubes with vertices on $\frac{1}{\lambda}\Z^n$.

As a consequence of \eqref{eq:ApproachingVolByDilationsWithG} and \eqref{eq:ApproachingIntegralWithG}, many continuous inequalities can be recovered from discrete inequalities. The following consequence of these equations will be useful in the sequel.

\begin{lemma}\label{lemma:limite} Let $K\subset \mathbb R^n$ be a convex body with $0\in K$, $m\in \mathbb N$ and $q\geq 1$. Then
	\begin{eqnarray*}
		&&\lim_{\lambda\to\infty}\int_{O(n)}\left(\frac{1}{(G_n(\lambda UK))^m}\sum_{x_1,\dots,x_m\in \lambda UK\cap\Z^n}\max_{1\leq i\leq m}\left|\left\langle \frac{x_i}{\lambda},e_n\right\rangle\right|^q\right)^{1/q}d\nu(U)\cr
		&=&\int_{O(n)}\left(\frac{1}{|UK|^m}\int_{UK}\dots\int_{UK}\max_{1\leq i\leq m}|\langle x_i,e_n\rangle|^q dx_m\dots dx_1\right)^{1/q}d\nu(U).
	\end{eqnarray*}
\end{lemma}
\begin{proof}
	For every $U\in O(n)$,
	\begin{eqnarray*}
		&&\lim_{\lambda\to\infty}\frac{1}{(G_n(\lambda UK))^m}\sum_{x_1,\dots,x_m\in \lambda UK\cap\Z^n}\max_{1\leq i\leq m}\left|\left\langle \frac{x_i}{\lambda},e_n\right\rangle\right|^q\cr
		&=&\lim_{\lambda\to\infty}\left(\frac{\lambda^n}{G_n(\lambda UK)}\right)^m\frac{1}{\lambda^{nm}}\sum_{x_1,\dots,x_N\in UK\cap\left(\frac{1}{\lambda}\Z^n\right)}\max_{1\leq i\leq m}\left|\left\langle x_i,e_n\right\rangle\right|^q\cr
		&=&\frac{1}{|UK|^m}\int_{UK}\dots\int_{UK}\max_{1\leq i\leq m}|\langle x_i,e_n\rangle|^q dx_m\dots dx_1.
	\end{eqnarray*}
	Besides, for every $\lambda>0$ and any $U\in O(n)$,
	\begin{eqnarray*}
		&&\left(\frac{1}{(G_n(\lambda UK))^m}\sum_{x_1,\dots,x_N\in \lambda UK\cap\left(\Z^n\right)}\max_{1\leq i\leq m}\left|\left\langle \frac{x_i}{\lambda},e_n\right\rangle\right|^q\right)^{\frac{1}{q}}\cr
&=&\left(\frac{1}{(G_n(\lambda UK))^m}\sum_{x_1,\dots,x_N\in UK\cap\left(\frac{1}{\lambda}\Z^n\right)}\max_{1\leq i\leq m}\left|\left\langle x_i,e_n\right\rangle\right|^q\right)^{\frac{1}{q}}\cr
		&\leq&\left(\frac{1}{(G_n(\lambda UK))^m}\sum_{x_1,\dots,x_m\in UK\cap\left(\frac{1}{\lambda}\Z^n\right)}\max_{x\in UK}|\langle x,e_n\rangle|^q\right)^{\frac{1}{q}}\\
		&=&\max_{x\in UK}|\langle x,e_n\rangle|=\max\{h_{UK}(e_n),h_{UK}(-e_n)\}\leq R(UK)=R(K),
	\end{eqnarray*}
where $R(L):=\inf\{R>0\,:\,L\subseteq RB_2^n\}$ denotes the circumradius of $L$, which is integrable in $O(n)$. By the dominated convergence theorem, the result follows. 	
\end{proof}

Notice also that given a Borel set $L\subseteq\R^n$, for every $x\in L\cap\Z^n$ we have that $x+\frac{1}{2}B_\infty^n$ is a cube with volume $|x+\frac{1}{2}B_\infty^n|=1$ and if $x_1,x_2\in L\cap\Z^n$ with $x_1\neq x_2$, then  $(x_1+\frac{1}{2}B_\infty^n)\cap(x_2+\frac{1}{2}B_\infty^n)$ has volume 0. Taking also into account that $L\cap\Z^n\subseteq L$ and then $(L\cap\Z^n)+\frac{1}{2}B_\infty^n\subseteq L+\frac{1}{2}B_\infty^n$ we have that
\begin{equation}\label{eq:LatticePointAndVolume}
G_n(L)=\left|(L\cap\Z^n)+\frac{1}{2}B_\infty^n\right|\leq\left|L+\frac{1}{2}B_\infty^n\right|.
\end{equation}

\subsection{Expected mean width of random polytopes}\label{subsec:ExpectedMeanWidthRandomPOlytopes}

In this section we give a proof of the lower bound in \eqref{eq:MeanWidthContinuous}, which holds for every $N\geq n$, if $N\geq 3$, since we could not find this proof in the literature. We state it here as a theorem.

\begin{thm}
Let $K\subseteq\R^n$ be a convex body with $0\in\operatorname{int} K$, $N\geq n$ with $N\geq 3$, let $X_1,\ldots, X_N$ be independent random vectors uniformly distributed on $K$, and let $K_N$ be the centrally symmetric random polytope generated by  $X_1,\dots X_N$, that is, $K_N=\operatorname{conv}\{\pm X_1,\ldots,\pm X_N\}$. There exists an absolute constant $c_1>0$ such that
$$
\E w(K_N)\geq c_1 w(Z_{\log N}(K)).
$$
\end{thm}

\begin{proof}
For any $\theta\in S^{n-1}$ and any $\alpha>0$ we have
\begin{eqnarray*}
\Pro\left(\max_{1\leq i\leq N}|\langle X_i,\theta\rangle|\geq\alpha\right)&=&1-\Pro\left(\max_{1\leq i\leq N}|\langle X_i,\theta\rangle|<\alpha\right)\\
&=&1-\Pro\left(|\langle X_1,\theta\rangle|<\alpha\right)^N\\
&=&1-(1-\Pro(|\langle X_1,\theta\rangle|\geq\alpha))^N.
\end{eqnarray*}
Therefore, taking $\alpha= h_{K_{\frac{1}{N}}}(\theta)$, where for any $0<\delta<1$, $K_\delta$ is the floating body defined by
$$
h_{K_\delta}(\theta)=\sup\{t>0\,:\,\Pro(|\langle X_1,\theta\rangle|\geq t)\geq \delta\}
$$
we have that
$$
\Pro\left(\max_{1\leq i\leq N}|\langle X_i,\theta\rangle|\geq h_{K_{\frac{1}{N}}}(\theta)\right)\geq 1-\left(1-\frac{1}{N}\right)^N\geq 1-\frac{1}{e}.
$$
Thus, by Markov's inequality, for any $\theta\in S^{n-1}$, calling $c=1-e^{-1}$,
$$
\E h_{K_N}(\theta)\geq h_{K_{\frac{1}{N}}}(\theta)\cdot \Pro\left(\max_{1\leq i\leq N}|\langle X_i,\theta\rangle|\geq h_{K_{\frac{1}{N}}}(\theta)\right)\geq ch_{K_{\frac{1}{N}}}(\theta).
$$
It was proved in \cite[Thm 2.2]{PW} that there exist absolute constants $c_1,c_2>0$ such that
$$
c_1 h_{Z_{\log\frac{e}{2\delta}}(K)}(\theta)\leq h_{K_\delta}(\theta)\leq c_2 h_{Z_{\log\frac{e}{2\delta}}(K)}(\theta).
$$
for any $0<\delta<\frac{1}{2}$. Since there exists an absolute constant $C>0$ such that $1\leq\log\frac{1}{\delta}\leq\log\frac{e}{2\delta}\leq C\log\frac{1}{\delta}$ for every $0<\delta\leq\frac{1}{e}<\frac{1}{2}$, the inclusion relation \eqref{eq:HolderAndBorell} implies that there exist absolute constants $c^\prime=c_1$ and $c^{\prime\prime}=c_2C>0$ such that
$$
c^\prime h_{Z_{\log\frac{1}{\delta}}(K)}(\theta)\leq h_{K_\delta}(\theta)\leq c^{\prime\prime} h_{Z_{\log\frac{1}{\delta}}(K)}(\theta).
$$
for any $0<\delta\leq\frac{1}{e}$. In particular, for any $\theta\in S^{n-1}$,
$$
\E h_{K_N}(\theta)\geq ch_{K_{\frac{1}{N}}}(\theta)\geq cc^{\prime}h_{Z_{\log N}(K)}(\theta).
$$
 if $N\geq 3$. Integrating in $\theta\in S^{n-1}$ and using Fubini's theorem we obtain the result with $c_1=cc^\prime$.
\end{proof}

\section{Discrete Borell's inequalities}

In this section we are going to prove Theorem \ref{thm:DiscreteBorell}. As a consequence of Theorem \ref{thm:DiscreteBorell} and Lemma \ref{lem:LimitC}, we can obtain inequality \eqref{eq:ContinuousBorell} for convex bodies with $0\in\textrm{int}K$.

\begin{proof}[Proof of Theorem \ref{thm:DiscreteBorell}]
The left-hand side inequality in \eqref{eq:DiscreteBorell} is a direct consequence of H\"older's inequality. Let us prove the right-hand side inequality. Fix $1\leq p\leq q$ and let $X$ be a random vector uniformly distributed on $K\cap\Z^n$ and let $Y$ be a random vector uniformly distributed on $(K+C_n)\cap\Z^n$. Let also, for $s>1$ to be chosen later,
\begin{eqnarray*}
A(s)&=&\left\{x\in\R^n\,:\,|\langle x,e_n\rangle|\leq s\left(\E|\langle X,e_n\rangle|^p\right)^\frac{1}{p}\right\}\cr
&=&\left\{x\in\R^n\,:\,|\langle x,e_n\rangle|^p\leq s^p\E|\langle X,e_n\rangle|^p\right\}.\cr
\end{eqnarray*}
Notice that, for any $t\geq1$ we have
$$
(tA(s))^c=A(ts)^c=\left\{x\in\R^n\,:\,|\langle x,e_n\rangle|> ts\left(\E|\langle X,e_n\rangle|^p\right)^\frac{1}{p}\right\}
$$
and
$$
A(s)^c+C_n=\left\{x\in\R^n\,:\,|\langle x,e_n\rangle|> s\left(\E|\langle X,e_n\rangle|^p\right)^\frac{1}{p}-1\right\}.
$$
Since $A(s)$ is convex and centrally symmetric, we have that for every $t>1$
$$
A(s)^c\supseteq\frac{2}{t+1}A(ts)^c+\frac{t-1}{t+1}A(s)
$$
and then, since $K$ is convex,
\begin{eqnarray*}
A(s)^c\cap K&\supseteq&\left(\frac{2}{t+1}A(ts)^c+\frac{t-1}{t+1}A(s)\right)\cap K\cr
&\supseteq&\frac{2}{t+1}(A(ts)^c\cap K)+\frac{t-1}{t+1}(A(s)\cap K).
\end{eqnarray*}
Therefore,
\begin{align*}
(A(s)^c+C_n)\cap (K+C_n)&\supseteq(A(s)^c\cap K+C_n)\cr
&\supseteq\frac{2}{t+1}(A(ts)^c\cap K)+\frac{t-1}{t+1}(A(s)\cap K)+C_n.
\end{align*}
By the discrete Brunn-Minkowski inequality (see Theorem \ref{thm: BM_lattice_point_no_G(K)G(L)>0}), if $A(ts)^c\cap K\neq\emptyset$,
\begin{eqnarray*}
G_n((A(s)^c+C_n)\cap (K+C_n))^\frac{1}{n}&\geq&\frac{2}{t+1} G_n((tA(s))^c\cap K)^\frac{1}{n}\cr
&+&\frac{t-1}{t+1}G_n(A(s)\cap K)^\frac{1}{n}.
\end{eqnarray*}
Thus,
$$
G_n((A(s)^c+C_n)\cap (K+C_n))\geq G_n(A(ts)^c\cap K)^\frac{2}{t+1}G_n(A(s)\cap K)^\frac{t-1}{t+1},
$$
being this inequality also true if $A(ts)^c\cap K=\emptyset$. Therefore, for any $t>1$
$$
\frac{G_n((A(s)^c+C_n)\cap (K+C_n))}{G_n(K)}\geq \left(\frac{G_n(A(ts)^c\cap K)}{G_n(K)}\right)^\frac{2}{t+1}\left(\frac{G_n(A(s)\cap K)}{G_n(K)}\right)^\frac{t-1}{t+1}.
$$

Notice that, by Markov's inequality,
\begin{align*}
	\frac{G_n(A(s)\cap K)}{G_n(K)}=1-\frac{G_n(A(s)^c\cap K)}{G_n(K)} \geq 1-s^{-p}.
\end{align*}	
Besides, also by  Markov's inequality, if $s\left(\E|\langle X,e_n\rangle|^p\right)^\frac{1}{p}>1$, then
$$
\frac{G_n((A(s)^c+C_n)\cap (K+C_n))}{G_n(K+C_n)}\leq \frac{\E|\langle Y,e_n\rangle|^p}{\left(s\left(\E|\langle X,e_n\rangle|^p\right)^\frac{1}{p}-1\right)^p}.
$$
Therefore, for every $t>1$, if $s\left(\E|\langle X,e_n\rangle|^p\right)^\frac{1}{p}>1$ we have
$$
\frac{G_n(K+C_n)}{G_n(K)}\frac{\E|\langle Y,e_n\rangle|^p}{\left(s\left(\E|\langle X,e_n\rangle|^p\right)^\frac{1}{p}-1\right)^p}\geq\left(\frac{G_n(A(ts)^c\cap K)}{G_n(K)}\right)^\frac{2}{t+1}\left(1-s^{-p}\right)^\frac{t-1}{t+1}
$$
and then, if also $s\geq 2$,
\begin{align*}
	\frac{G_n(A(ts)^c\cap K)}{G_n(K)}
&\leq \left(1-s^{-p}\right)\left(\frac{G_n(K+C_n)}{G_n(K)}\frac{\E|\langle Y,e_n\rangle|^p}{\left(1-s^{-p}\right)\left(s\left(\E|\langle X,e_n\rangle|^p\right)^\frac{1}{p}-1\right)^p}\right)^{\!\frac{t+1}{2}}\\
&\leq\left(\left(\frac{G_n(K+C_n)}{G_n(K)}\right)^\frac{1}{p}\frac{(\E|\langle Y,e_n\rangle|^p)^\frac{1}{p}}{\left(1-s^{-p}\right)^\frac{1}{p}\left(s\left(\E|\langle X,e_n\rangle|^p\right)^\frac{1}{p}-1\right)}\right)^{\!\frac{p(t+1)}{2}}\\
&\leq\left(\left(\frac{G_n(K+C_n)}{G_n(K)}\right)^\frac{1}{p}\frac{2(\E|\langle Y,e_n\rangle|^p)^\frac{1}{p}}{s\left(\E|\langle X,e_n\rangle|^p\right)^\frac{1}{p}-1}\right)^{\!\frac{p(t+1)}{2}}.\cr
\end{align*}

Now, choose

$$
s=\frac{4+4(\E|\langle Y,e_n\rangle|^p)^\frac{1}{p}\left(\frac{G_n(K+C_n)}{G_n(K)}\right)^\frac{1}{p}}{\left(\E|\langle X,e_n\rangle|^p\right)^\frac{1}{p}}=4D_0(K,p).
$$

Recall that, as seen in Remark \ref{rmk:LowerboundC0},  $D_0(K, p)\geq 1$ and so $s\geq 4$. Also,
$$
s\left(\E|\langle X,e_n\rangle|^p\right)^\frac{1}{p}-1\geq 4(\E|\langle Y,e_n\rangle|^p)^\frac{1}{p}\left(\frac{G_n(K+C_n)}{G_n(K)}\right)^\frac{1}{p}>0.
$$

Thus, we have that
\begin{eqnarray*}
\Pro\left(|\langle X, e_n\rangle|>ts\left(\E|\langle X, e_n\rangle|^p\right)^\frac{1}{p}\right)&=&\frac{G_n(A(ts)^c\cap K)}{G_n(K)}\leq\left(\frac{1}{2}\right)^\frac{p(t+1)}{2}\leq\left(\frac{1}{2}\right)^\frac{pt}{2}\cr
&=&e^{-\frac{pt\log 2}{2}}.
\end{eqnarray*}

Therefore,

\allowdisplaybreaks
\begin{align*}
	\frac{\E|\langle X, e_n\rangle|^q}{\left(\E|\langle X, e_n\rangle|^p\right)^\frac{q}{p}}&=\frac{q\int_0^\infty t^{q-1}\Pro\left(|\langle X, e_n\rangle|>t\right)dt}{\left(\E|\langle X, e_n\rangle|^p\right)^\frac{q}{p}}\cr
	&=q\int_0^\infty u^{q-1}\Pro\left(|\langle X, e_n\rangle|>u\left(\E|\langle X, e_n\rangle|^p\right)^\frac{1}{p}\right)du\cr
	&=q\int_0^{s}u^{q-1}\Pro\left(|\langle X, e_n\rangle|>u\left(\E|\langle X, e_n\rangle|^p\right)^\frac{1}{p}\right)du\\
	& \quad +q\int_{s}^\infty u^{q-1}\Pro\left(|\langle X, e_n\rangle|>u\left(\E|\langle X, e_n\rangle|^p\right)^\frac{1}{p}\right)du\cr
	&\leq s^q + s^q q\int_{1}^\infty v^{q-1}\Pro\left(|\langle X, e_n\rangle|>vs\left(\E|\langle X, e_n\rangle|^p\right)^\frac{1}{p}\right)dv\cr
	&\leq s^q +s^q q\int_{1}^\infty v^{q-1}e^{-\frac{pv\log 2}{2}}dv \cr
	&\leq s^q + s^q q\int_{0}^\infty v^{q-1}e^{-\frac{pv\log 2}{2}}dv\cr
	&=s^q +s^q\left(\frac{2}{p\log 2}\right)^qq\int_{0}^\infty r^{q-1}e^{-r}dr\cr
	&=s^q + s^q\left(\frac{2}{p\log 2}\right)^q\Gamma(1+q)\cr
	&=s^q\left(1+\left(\frac{2}{p\log 2}\right)^q\Gamma(1+q)\right).
\end{align*}

Taking into account that, by Stirling's formula, there exists an absolute constant $K_0>0$ such that $\frac{2}{\log 2}\Gamma(1+q)^\frac{1}{q}\leq K_0q$ for any $q\geq 1$, we obtain
\begin{eqnarray*}
\left(\E|\langle X, e_n\rangle|^q\right)^\frac{1}{q}&\leq&s\left(\E|\langle X, e_n\rangle|^p\right)^\frac{1}{p}\left(1+\left(\frac{2}{p\log 2}\right)^q\Gamma(1+q)\right)^\frac{1}{q}\cr
&\leq&s\left(\E|\langle X, e_n\rangle|^p\right)^\frac{1}{p}\left(1+\left(\frac{2}{p\log 2}\right)\Gamma(1+q)^\frac{1}{q}\right)\cr
&\leq&s\left(\E|\langle X, e_n\rangle|^p\right)^\frac{1}{p}\left(1+K_0\frac{q}{p}\right)\cr
&\leq&s\left(\E|\langle X, e_n\rangle|^p\right)^\frac{1}{p}(K_0+1)\frac{q}{p}\cr
&=&4(K_0+1)\frac{q}{p}D_0(K,p),
\end{eqnarray*}
which proves  inequality \eqref{eq:DiscreteBorell} with $C=4(K_0+1)$.
\end{proof}

Given a convex body $K$ with $0\in \operatorname{int}K$, we denote $r(K)$ and $R(K)$ the inradius and circumradius of $K$, that is, the optimal constants satisfying that $r(K)B_2^n\subseteq K \subseteq R(K)B_2^n$.

\begin{lemma}\label{lemma:C0bdd}
	Let $K\subset \mathbb R^n$ be a convex body with $0\in \operatorname{int} K$,
	 $\max_{x\in K\cap \mathbb Z^n}|\langle x,e_n\rangle |\geq 1$ and  $\sqrt{n}R(K)\leq 2r(K)^2$. Then,
	
	\begin{itemize}
		\item [a)] $G_n((1+t)K\setminus K)\leq \left(\left(1+t+\frac{\sqrt{n}}{2r(K)}\right)^{\! n}-\left(1-\frac{\sqrt{n}R(K)}{2r^2(K)}\right)^{\!n}\right)|K|$ for all $t>0$.
		\item [b)] For any $p\geq 1$ we have
		\begin{align*}D_0(K, p)&\leq 1+\frac{|K+\frac{\sqrt{n}}{2}B_2^n|}{\sum_{x\in r(K)B_2^n\cap \Z^n} |\langle x, e_n\rangle|}\\
		&\quad +\frac{\left(1+\frac{\sqrt{n}}{r(K)}\right)R(K)\,\left(\left(1+\frac{3\sqrt{n}}{2r(K)}\right)^n-\left(1-\frac{\sqrt{n}R(K)}{2r^2(K)}\right)^n\right)^{\frac{1}{p}}|K|^\frac{1}{p}}{\left(\frac{1}{G_n(r(K)B_2^n)}\sum_{x\in r(K)B_2^n\cap\Z^n}|\langle x,e_n\rangle|\right) G_n( r(K)B_2^n)^\frac{1}{p}}. \end{align*}
	\end{itemize}	
\end{lemma}

\begin{proof}
Let us first prove a).	By \eqref{eq:LatticePointAndVolume},
\begin{align*}
	|G_n((1+t)K\setminus K)|&\leq \left|(1+t)K\setminus K + \frac{1}{2}B_\infty^n\right|\leq \left|(1+t)K\setminus K+\frac{\sqrt{n}}{2}B_2^n\right|\\
	&\leq \left|(1+t)K\setminus K+\frac{\sqrt{n}}{2r(K)}K\right|.
\end{align*}
Now, since
$$
-K\subseteq R(K)B_2^n\subseteq\frac{R(K)}{r(K)}K,
$$
we have
$$
-\frac{\sqrt{n}}{2r(K)}K\subseteq\frac{\sqrt{n}R(K)}{2r^2(K)}K.
$$
Recall that $\frac{\sqrt{n}R(K)}{2r^2(K)}\leq 1$. It follows that
\begin{align*}
(1+t)K\setminus K+\frac{\sqrt{n}}{2r(K)}K&\subseteq \left(1+t+\frac{\sqrt{n}}{2r(K)}\right)K\setminus\left(1-\frac{\sqrt{n}R(K)}{2r^2(K)}\right)K\\
\end{align*}
and then
\begin{align*}
	\left|(1+t)K\setminus K+\frac{\sqrt{n}}{2r(K)}K\right|&\leq \left|\left(1+t+\frac{\sqrt{n}}{2r(K)}\right)K\setminus\left(1-\frac{\sqrt{n}R(K)}{2r^2(K)}\right)K\right|\\
	&= \left(\left(1+t+\frac{\sqrt{n}}{2r(K)}\right)^{\! n}-\left(1-\frac{\sqrt{n}R(K)}{2r^2(K)}\right)^{\! n}\right)|K|.\\
\end{align*}

Let us now prove b). We have that if $X$ and $Y$ are random vectors uniformly distributed on $K\cap\Z^n$ and $(K+C_n)\cap\Z^n$ respectively, for any $p\geq1$,
\begin{align*}
	D_0(K, p) &=\frac{1+\left(\E|\langle Y, e_n\rangle|^p\right)^\frac{1}{p}\left(\frac{G_n(K+C_n)}{G_n(K)}\right)^\frac{1}{p}}{\left(\E|\langle X, e_n\rangle|^p\right)^\frac{1}{p}}\\
	&=\frac{1}{\left(\E|\langle X, e_n\rangle|^p\right)^\frac{1}{p}}+\left(\frac{\sum_{x\in(K+C_n)\cap\Z^n}|\langle x, e_n\rangle|^p}{\sum_{x\in K\cap\Z^n}|\langle x,e_n\rangle|^p}\right)^\frac{1}{p}
\end{align*}
On the one hand, by H\"older's inequality and \eqref{eq:LatticePointAndVolume},
\begin{align*}
	\left(\E|\langle X, e_n\rangle|^p\right)^\frac{1}{p} &\geq \E|\langle X, e_n\rangle|= \frac{\sum_{x\in K\cap \Z^n} |\langle x, e_n\rangle|}{G_n(K)}\geq \frac{\sum_{x\in r(K)B_2^n\cap \Z^n} |\langle x, e_n\rangle|}{|K+\frac{1}{2}B_\infty^n|}\\
	&\geq \frac{\sum_{x\in r(K)B_2^n\cap \Z^n} |\langle x, e_n\rangle|}{|K+\frac{\sqrt{n}}{2}B_2^n|}.
\end{align*}	
On the other hand,
\allowdisplaybreaks[4]
\begin{align*}
	&\left(\frac{\sum_{x\in(K+C_n)\cap\Z^n}|\langle x, e_n\rangle|^p}{\sum_{x\in K\cap\Z^n}|\langle x,e_n\rangle|^p}\right)^{\!\frac{1}{p}}=\left(1+\frac{\sum_{x\in((K+C_n)\setminus K)\cap\Z^n}|\langle x, e_n\rangle|^p}{\sum_{x\in K\cap\Z^n}|\langle x,e_n\rangle|^p}\right)^{\!\frac{1}{p}}\\
	&\qquad\qquad\qquad \leq 1+ \left(\frac{\sum_{x\in((K+C_n)\setminus K)\cap\Z^n}|\langle x, e_n\rangle|^p}{\sum_{x\in K\cap\Z^n}|\langle x,e_n\rangle|^p}\right)^{\!\frac{1}{p}}\\
	&\qquad\qquad\qquad \leq 1+ \left(\frac{\sum_{x\in((K+\sqrt{n}B_2^n)\setminus K)\cap\Z^n}|\langle x, e_n\rangle|^p}{\sum_{x\in r(K)B_2^n\cap\Z^n}|\langle x,e_n\rangle|^p}\right)^{\!\frac{1}{p}}\\
	&\qquad\qquad\qquad \leq 1+ \left(\frac{\sum_{x\in((K+\frac{\sqrt{n}}{r(K)}K)\setminus K)\cap\Z^n}|\langle x, e_n\rangle|^p}{\sum_{x\in r(K)B_2^n\cap\Z^n}|\langle x,e_n\rangle|^p}\right)^{\!\frac{1}{p}}\\
	&\qquad\qquad\qquad \leq 1+\frac{h_{(1+\frac{\sqrt{n}}{r(K)})K}(e_n)\,G_n\left(\left(1+\frac{\sqrt{n}}{r(K)}\right)K\setminus K\right)^{\frac{1}{p}}}{\left(\sum_{x\in r(K)B_2^n\cap\Z^n}|\langle x,e_n\rangle|^p\right)^\frac{1}{p}}\\
&\qquad\qquad\qquad \leq 1+\frac{\left(1+\frac{\sqrt{n}}{r(K)}\right)R(K)\,G_n\left(\left(1+\frac{\sqrt{n}}{r(K)}\right)K\setminus K\right)^\frac{1}{p}}{\left(\frac{1}{G_n(r(K)B_2^n)}\sum_{x\in r(K)B_2^n\cap\Z^n}|\langle x,e_n\rangle|^p\right)^{\frac{1}{p}} G_n( r(K)B_2^n)^\frac{1}{p}}\\
&\qquad\qquad\qquad \leq 1+\frac{\left(1+\frac{\sqrt{n}}{r(K)}\right)R(K)\,G_n\left(\left(1+\frac{\sqrt{n}}{r(K)}\right)K\setminus K\right)^\frac{1}{p}}{\left(\frac{1}{G_n(r(K)B_2^n)}\sum_{x\in r(K)B_2^n\cap\Z^n}|\langle x,e_n\rangle|\right) G_n( r(K)B_2^n)^\frac{1}{p}}.
\end{align*}	
By a), we get
\begin{align*}
	\left|G_n\left(\left(1+\frac{\sqrt{n}}{r(K)}\right)K\setminus K\right)\right|\leq \left(\left(1+\frac{3\sqrt{n}}{2r(K)}\right)^n-\left(1-\frac{\sqrt{n}R(K)}{2r^2(K)}\right)^n\right)|K|
\end{align*}
and the statement follows.
\end{proof}

Let us now prove Lemma \ref{lem:LimitC}, i.e., that, for every convex body $K\subseteq\R^n$ with $0\in\operatorname{int}(K)$ and any $q\geq1$, we have
$$
\lim_{\lambda\to\infty}D_{q}(\lambda K)=1.
$$

\begin{proof}[Proof of Lemma \ref{lem:LimitC}]

Let $K\subseteq\R^n$ be a convex body with $0\in\operatorname{int}(K)$, $p\geq 1$ and $U\in O(n)$. Notice that if $\lambda\geq r(K)^{-1}$ we have $B_2^n\subset \lambda UK$
 and then $\displaystyle{\max_{x\in\lambda UK\cap\Z^n}|\langle x,e_n\rangle|\geq 1}$.
 Therefore, if $\lambda\geq r(K)^{-1}$ the number $D_0(\lambda UK, p)$ is well defined. Since $r(\lambda UK)=\lambda r(K)$, $R(\lambda UK)=\lambda R(K)$ and $|\lambda UK|=\lambda^n|K|$, by Lemma \ref{lemma:C0bdd} we have that, whenever $\lambda\geq\max\left\{r(K)^{-1},\frac{\sqrt{n}R(K)}{2r^2(K)}\right\}$,

\begin{align*}
	1\leq D_0(\lambda UK, p) &\leq  1+\frac{|\lambda UK+\frac{\sqrt{n}}{2}B_2^n|}{\sum_{x\in \lambda r(K)B_2^n\cap \Z^n} |\langle x, e_n\rangle|}\\
	&\quad +\frac{\left(\lambda+\frac{\sqrt{n}}{r(K)}\right) R(K)\,\left(\left(\lambda+\frac{3\sqrt{n}}{2r(K)}\right)^n-\left(\lambda-\frac{\sqrt{n}R(K)}{2r^2(K)}\right)^n\right)^{\frac{1}{p}}|K|^\frac{1}{p}}{\left(\frac{1}{G_n(\lambda r(K)B_2^n)}\sum_{x\in \lambda r(K)B_2^n\cap\Z^n}|\langle x,e_n\rangle|\right) G_n(\lambda r(K)B_2^n)^\frac{1}{p}}.\\
\end{align*}
Thus, for every $\lambda\geq\max\left\{r(K)^{-1},\frac{\sqrt{n}R(K)}{2r^2(K)}\right\}$,

\begin{align*}1\leq D_{q}(\lambda K)&=\sup_{\subs{U\in O(n)}{1\leq p\leq q}}D_0(\lambda UK,p)\\
	&\leq 	1+\sup_{U\in O(n)}\frac{|\lambda UK+\frac{\sqrt{n}}{2}B_2^n|}{\sum_{x\in \lambda r(K)B_2^n\cap \Z^n} |\langle x, e_n\rangle|}\\
	&+\sup_{\subs{U\in O(n)}{1\leq p\leq q}}\frac{\left(\lambda+\frac{\sqrt{n}}{r(K)}\right) R(K)\,\left(\left(\lambda+\frac{3\sqrt{n}}{2r(K)}\right)^n-\left(\lambda-\frac{\sqrt{n}R(K)}{2r^2(K)}\right)^n\right)^{\frac{1}{p}}|K|^\frac{1}{p}}{\left(\frac{1}{G_n(\lambda r(K)B_2^n)}\sum_{x\in \lambda r(K)B_2^n\cap\Z^n}|\langle x,e_n\rangle|\right) G_n(\lambda r(K)B_2^n)^\frac{1}{p}}.
\end{align*}

Let us see that both supremums tend to $0$ as $\lambda\to\infty$. First,

\begin{eqnarray*}
	\sup_{U\in O(n)}\frac{|\lambda UK+\frac{\sqrt{n}}{2}B_2^n|}{\sum_{x\in \lambda r(K)B_2^n\cap \Z^n} |\langle x, e_n\rangle|}&=&\sup_{U\in O(n)}\frac{|\lambda U(K+\frac{\sqrt{n}}{2}B_2^n)|}{\sum_{x\in \lambda r(K)B_2^n\cap \Z^n} |\langle x, e_n\rangle|}\\
	&=&
	\frac{\lambda^n\left| K+\frac{\sqrt{n}}{2\lambda}B_2^n\right|}{\sum_{x\in \lambda r(K)B_2^n\cap\Z^n}|\langle x,e_n\rangle|}\\
	&=&\frac{\frac{1}{\lambda}\left|K+\frac{\sqrt{n}}{2\lambda}B_2^n\right|}{\frac{1}{\lambda^n}\sum_{x\in \lambda r(K)B_2^n\cap\Z^n}|\langle\frac{x}{\lambda},e_n\rangle|}.\cr
\end{eqnarray*}

Since $\displaystyle{\lim_{\lambda\to\infty}\left|K+\frac{\sqrt{n}}{2\lambda}B_2^n\right|=|K|}$ and, by \eqref{eq:ApproachingIntegralWithG},
\begin{align*}\lim_{\lambda\to\infty}\frac{1}{\lambda^n}\sum_{x\in\lambda r(K)B_2^n\cap\Z^n}\left|\left\langle \frac{x}{\lambda},e_n\right\rangle\right|&=\lim_{\lambda\to\infty}\frac{1}{\lambda^n}\sum_{x\in r(K)B_2^n\cap\frac{1}{\lambda}\Z^n}|\langle x,e_n\rangle|\\
	&=\int_{r(K)B_2^n}|\langle x,e_n\rangle|dx,\end{align*}
we get that
$$
\lim_{\lambda\to\infty}\frac{\frac{1}{\lambda}\left|K+\frac{\sqrt{n}}{2\lambda}B_2^n\right|}{\frac{1}{\lambda^n}\sum_{x\in r(K)B_2^n\cap\frac{1}{\lambda}\Z^n}|\langle x,e_n\rangle|}=0
$$
and then
$$
\lim_{\lambda\to\infty}\sup_{U\in O(n)}\frac{|\lambda UK+\frac{\sqrt{n}}{2}B_2^n|}{\sum_{x\in \lambda r(K)B_2^n\cap\Z^n}|\langle x,e_n\rangle|}=0.
$$
Let us now see that
$$
\lim_{\lambda\to\infty}\sup_{\subs{U\in O(n)}{1\leq p\leq q}}\frac{\left(\lambda+\frac{\sqrt{n}}{r(K)}\right) R(K)\,\left(\left(\lambda+\frac{3\sqrt{n}}{2r(K)}\right)^n-\left(\lambda-\frac{\sqrt{n}R(K)}{2r^2(K)}\right)^n\right)^{\frac{1}{p}}|K|^\frac{1}{p}}{\left(\frac{1}{G_n(\lambda r(K)B_2^n)}\sum_{x\in \lambda r(K)B_2^n\cap\Z^n}|\langle x,e_n\rangle|\right) G_n(\lambda r(K)B_2^n)^\frac{1}{p}}=0.
$$
Since
\begin{eqnarray*}
&&\lim_{\lambda\to\infty}\frac{\left(\left(\lambda+\frac{3\sqrt{n}}{2r(K)}\right)^n-\left(\lambda-\frac{\sqrt{n}R(K)}{2r^2(K)}\right)^n\right)|K|}{G_n\left(\lambda r(K)B_2^n\right)}\cr
&=&\lim_{\lambda\to\infty}\frac{\frac{1}{\lambda^n}\left(\left(\lambda+\frac{3\sqrt{n}}{2r(K)}\right)^n-\left(\lambda-\frac{\sqrt{n}R(K)}{2r^2(K)}\right)^n\right)|K|}{\frac{1}{\lambda^n}G_n\left(\lambda r(K)B_2^n\right)}\cr
&=&\frac{0}{|r(K)B_2^n|}=0,
\end{eqnarray*}
there exists $\lambda_1>0$ such that if $\lambda\geq\lambda_1$ then $\frac{\left(\left(\lambda+\frac{3\sqrt{n}}{2r(K)}\right)^n-\left(\lambda-\frac{\sqrt{n}R(K)}{2r^2(K)}\right)^n\right)|K|}{G_n\left(\lambda r(K)B_2^n\right)}\leq1$ and, for every $1\leq p\leq q$
$$
\left(\frac{\left(\lambda+\frac{3\sqrt{n}}{2r(K)}\right)^n-\left(\lambda-\frac{\sqrt{n}R(K)}{2r^2(K)}\right)^n}{G_n\left(\lambda r(K)B_2^n\right)}|K|\right)^{\!\frac{1}{p}}\leq\left(\frac{\left(\lambda+\frac{3\sqrt{n}}{3r(K)}\right)^n-\left(\lambda-\frac{\sqrt{n}}{r(K)}\right)^n}{G_n\left(\lambda r(K)B_2^n\right)}|K|\right)^{\!\frac{1}{q}}.
$$
Therefore, for every $\lambda\geq\max\left\{r(K)^{-1},\frac{\sqrt{n}R(K)}{2r^2(K)},\lambda_1\right\}$,
\begin{eqnarray*}
&&\sup_{\subs{U\in O(n)}{1\leq p\leq q}}\frac{\left(\lambda+\frac{\sqrt{n}}{r(K)}\right) R(K)\,\left(\left(\lambda+\frac{3\sqrt{n}}{2r(K)}\right)^n-\left(\lambda-\frac{\sqrt{n}R(K)}{2r^2(K)}\right)^n\right)^{\frac{1}{p}}|K|^\frac{1}{p}}{\left(\frac{1}{G_n(\lambda r(K)B_2^n)}\sum_{x\in \lambda r(K)B_2^n\cap\Z^n}|\langle x,e_n\rangle|\right) G_n(\lambda r(K)B_2^n)^\frac{1}{p}}\\
&\leq&\frac{\left(1+\frac{\sqrt{n}}{\lambda r(K)}\right)R(K)}{\frac{1}{G_n(\lambda r(K)B_2^n)}\sum_{x\in  \lambda r(K)B_2^n\cap\Z^n}|\langle\frac{x}{\lambda},e_n\rangle|}
\left(\frac{\left(\lambda+\frac{3\sqrt{n}}{2r(K)}\right)^n-\left(\lambda-\frac{\sqrt{n}R(K)}{2r^2(K)}\right)^n}{G_n\left(\lambda r(K)B_2^n\right)}|K|\right)^{\!\frac{1}{q}}\\
&=&\frac{\left(1+\frac{\sqrt{n}}{\lambda r(K)}\right)R(K)}{\frac{1}{\frac{1}{\lambda^n}G_n(\lambda r(K)B_2^n)}\frac{1}{\lambda^n}\sum_{x\in  \lambda r(K)B_2^n\cap\Z^n}|\langle\frac{x}{\lambda},e_n\rangle|}\left(\frac{\left(\lambda+\frac{3\sqrt{n}}{2r(K)}\right)^n-\left(\lambda-\frac{\sqrt{n}R(K)}{2r^2(K)}\right)^n}{G_n\left(\lambda r(K)B_2^n\right)}|K|\right)^{\!\frac{1}{q}}.\cr
\end{eqnarray*}
Since
\begin{itemize}
\item $\displaystyle{\lim_{\lambda\to\infty}\left(1+\frac{\sqrt{n}}{\lambda r(K)}\right)=1}$,
\item $\displaystyle{\lim_{\lambda\to\infty}\frac{\left(\lambda+\frac{3\sqrt{n}}{2r(K)}\right)^n-\left(\lambda-\frac{\sqrt{n}R(K)}{2r^2(K)}\right)^n}{G_n\left(\lambda r(K)B_2^n\right)}|K|=0}$,
\item $\displaystyle{\lim_{\lambda\to\infty}\frac{1}{\lambda^n}\sum_{x\in \lambda r(K)B_2^n\cap\Z^n}\left|\left\langle \frac{x}{\lambda},e_n\right\rangle\right|=\int_{r(K)B_2^n}|\langle x,e_n\rangle|dx}$, and
\item $\displaystyle{\lim_{\lambda\to\infty}\frac{1}{\lambda^n}G_n(\lambda r(K)B_2^n)=|r(K)B_2^n|}$,
\end{itemize}
we have that
$$
\lim_{\lambda\to\infty}\frac{\left(1+\frac{\sqrt{n}}{\lambda r(K)}\right)R(K)}{\frac{1}{\frac{1}{\lambda^n}G_n(\lambda r(K)B_2^n)}\frac{1}{\lambda^n}\sum_{x\in  \lambda r(K)B_2^n\cap\Z^n}|\langle\frac{x}{\lambda},e_n\rangle|}\left(\frac{\left(\lambda+\frac{3\sqrt{n}}{3r(K)}\right)^n-\left(\lambda-\frac{\sqrt{n}}{r(K)}\right)^n}{G_n\left(\lambda r(K)B_2^n\right)}|K|\right)^{\!\frac{1}{q}}=0
$$
and then
\[\lim_{\lambda\to\infty}\sup_{\subs{U\in O(n)}{1\leq p\leq q}}\frac{\left(\lambda+\frac{\sqrt{n}}{r(K)}\right) R(K)\,\left(\left(\lambda+\frac{3\sqrt{n}}{2r(K)}\right)^n-\left(\lambda-\frac{\sqrt{n}R(K)}{2r^2(K)}\right)^n\right)^{\frac{1}{p}}|K|^\frac{1}{p}}{\left(\frac{1}{G_n(\lambda r(K)B_2^n)}\sum_{x\in \lambda r(K)B_2^n\cap\Z^n}|\langle x,e_n\rangle|\right) G_n(\lambda r(K)B_2^n)^\frac{1}{p}}=0.\]
Therefore,
$$
\lim_{\lambda\to\infty}D_{q}(\lambda K)=1.
$$
\end{proof}

As a consequence of Theorem \ref{thm:DiscreteBorell} and Lemma \ref{lem:LimitC}, we can obtain inequality \eqref{eq:ContinuousBorellOneDirection} for convex bodies containing the origin in its interior.

\begin{cor}\label{cor:Continuous Borell}
There exists an absolute constant $C>0$ such that for every convex body $K\subseteq\R^n$ with $0\in\operatorname{int}K$, if $X$ is a random vector uniformly distributed on $K$, then for any $1\leq p\leq q$
$$
\left(\E|\langle X,e_n\rangle|^p\right)^\frac{1}{p}\leq\left(\E|\langle X,e_n\rangle|^q\right)^\frac{1}{q}\leq C\frac{q}{p}\left(\E|\langle X,e_n\rangle|^p\right)^\frac{1}{p}.
$$
\end{cor}

\begin{proof}
Let $K\subseteq\R^n$ be a convex body with $0\in\operatorname{int}K$ and $1\leq p\leq q$. Let $\lambda\geq r(K)^{-1}>0$. Then, we have that
 $\displaystyle{\max_{x\in\lambda K\cap\Z^n}|\langle x,e_n\rangle|\geq1}$. By Theorem \ref{thm:DiscreteBorell}, if $X_\lambda$ is a random vector uniformly distributed on $\lambda K\cap\Z^n$,
$$
\left(\E|\langle X_\lambda, e_n\rangle|^p\right)^\frac{1}{p}\leq\left(\E|\langle X_\lambda, e_n\rangle|^q\right)^\frac{1}{q}\leq C\frac{q}{p}D_0(\lambda K,p)\left(\E|\langle X_\lambda, e_n\rangle|^p\right)^\frac{1}{p},
$$
where $C>0$ is an absolute constant.
Equivalently,
$$
\left(\E\left|\left\langle \frac{X_\lambda}{\lambda}, e_n\right\rangle\right|^p\right)^\frac{1}{p}\leq\left(\E\left|\left\langle \frac{X_\lambda}{\lambda}, e_n\right\rangle\right|^q\right)^\frac{1}{q}
\leq C\frac{q}{p}D_0(\lambda K,p)\left(\E\left|\left\langle \frac{X_\lambda}{\lambda}, e_n\right\rangle\right|^p\right)^\frac{1}{p}.
$$
On the one hand, for any $p\geq 1$ we have
\begin{eqnarray*}
\E\left|\left\langle \frac{X_\lambda}{\lambda}, e_n\right\rangle\right|^p&=&\frac{1}{G_n(\lambda K)}\sum_{x\in\lambda K\cap\Z^n}\left|\left\langle\frac{x}{\lambda},e_n\right\rangle\right|^p\\
&=&\frac{1}{\frac{1}{\lambda^n}G_n(\lambda K)}\frac{1}{\lambda^n}\sum_{x\in\lambda K\cap\Z^n}\left|\left\langle\frac{x}{\lambda},e_n\right\rangle\right|^p\cr
&=&\frac{1}{\frac{1}{\lambda^n}G_n(\lambda K)}\frac{1}{\lambda^n}\sum_{x\in K\cap\frac{1}{\lambda}\Z^n}\left|\left\langle x,e_n\right\rangle\right|^p,
\end{eqnarray*}
and, by \eqref{eq:ApproachingVolByDilationsWithG} and \eqref{eq:ApproachingIntegralWithG}, taking the limit as $\lambda\to\infty$ we obtain that
$$
\lim_{\lambda\to\infty}\E\left|\left\langle \frac{X_\lambda}{\lambda}, e_n\right\rangle\right|^p=\frac{1}{|K|}\int_{K}|\langle x,e_n\rangle|^pdx=\E|\langle X,e_n\rangle|^p,
$$
where $X$ is a random vector uniformly distributed on $K$. On the other hand, as seen in Remark \ref{rmk:LowerboundC0},
$$
\lim_{\lambda\to\infty} D_0(\lambda K,p)=1
$$
and then, if $X$ is a random vector uniformly distributed on $K$,
$$
\left(\E|\langle X,e_n\rangle|^p\right)^\frac{1}{p}\leq\left(\E|\langle X,e_n\rangle|^q\right)^\frac{1}{q}\leq C\frac{q}{p}\left(\E|\langle X,e_n\rangle|^p\right)^\frac{1}{p}.
$$
\end{proof}

\section{Expected mean width of random polytopes via discrete inequalities}

In this section we will prove Theorems \ref{thm:DiscreteUpperBoundMeanWidth} and \ref{thm:DiscreteLowerBoundMeanWidth} and show how they imply the inequalities in \eqref{eq:MeanWidthRotations} and, equivalently, \eqref{eq:MeanWidthContinuous}. In order to prove Theorem \ref{thm:DiscreteUpperBoundMeanWidth} we prove the following lemma.

\begin{lemma}\label{lem:ProbabilisticEstimateL}
Let $L\subseteq\R^n$ be a convex body with $0\in L$, $N\in\N$ and $a>0$. Let $X_1,\ldots,X_N$ be independent random vectors uniformly distributed on $L\cap\Z^n$. Then, for any $q>0$,
\[\Pro\left(\max_{1\leq i\leq N}|\langle X_i,e_n\rangle|\geq a\left(\E|\langle X_1,e_n\rangle|^q\right)^\frac{1}{q}\right)\leq Na^{-q}.
\]
\end{lemma}

\begin{proof}
By the union bound, we have that
\begin{align*}
\Pro\left(\max_{1\leq i\leq N}|\langle X_i,e_n\rangle|\geq a\left(\E|\langle X_1,e_n\rangle|^q\right)^\frac{1}{q}\right)
&\leq N\Pro\left(|\langle X_1,e_n\rangle|\geq a\left(\E|\langle X_1,e_n\rangle|^q\right)^\frac{1}{q}\right)\\
&=N\Pro\left(|\langle X_1,e_n\rangle|^q\geq a^q\E|\langle X_1,e_n\rangle|^q\right).
\end{align*}
By Markov's inequality
$$
\Pro\left(|\langle X_1,e_n\rangle|^q\geq a^q\E|\langle X_1,e_n\rangle|^q\right)\leq a^{-q}.
$$
Therefore,
$$
\Pro\left(\max_{1\leq i\leq N}|\langle X_i,e_n\rangle|\geq a\left(\E|\langle X_1,e_n\rangle|^q\right)^\frac{1}{q}\right)\leq Na^{-q}.
$$
\end{proof}

\begin{proof}[Proof of Theorem \ref{thm:DiscreteUpperBoundMeanWidth}]
For any $a>0$ and any $q>1$, let us call
$$
h_q(U):=\left(\frac{1}{G_n(UK)}\sum_{x\in UK\cap\Z^n}|\langle x,e_n\rangle|^q\right)^\frac{1}{q}=\left(\E|\langle X_U,e_n\rangle|^q\right)^\frac{1}{q},
$$
where $X_U$ is a random vector uniformly distributed on $UK\cap\Z^n$. Let also $A$ and $B$ be the following subsets of $(\Z^n)^N\times O(n)$:
\begin{align*}
A&=\left\{(x_1,\dots,x_N,U)\,:\,x_i\in UK\cap\Z^n,\,\forall 1\leq i\leq N,\,\max_{1\leq i\leq N}|\langle x_i,e_n\rangle|\geq ah_q(U)\right\},\\
B&=\left\{(x_1,\dots,x_N,U)\,:\,x_i\in UK\cap\Z^n,\,\forall 1\leq i\leq N,\,\max_{1\leq i\leq N}|\langle x_i,e_n\rangle|< ah_q(U)\right\}.
\end{align*}
Notice that $A\cap B=\emptyset$ and
$$
A\cup B=\{(x_1,\dots,x_N,U)\,:\,x_i\in UK\cap\Z^n,\,\forall 1\leq i\leq N\}.
$$
Therefore,
\begin{eqnarray*}
&&\int_{O(n)}\frac{1}{(G_n(UK))^N}\sum_{x_1,\dots,x_N\in UK\cap\Z^n}\max_{1\leq i\leq N}|\langle x_i,e_n\rangle|d\nu(U)\cr
&=&\int_{O(n)}\frac{1}{(G_n(UK))^N}\sum_{x_1,\dots,x_N\in UK\cap\Z^n}\max_{1\leq i\leq N}|\langle x_i,e_n\rangle|\chi_{A}(x_1,\dots,x_N,U)d\nu(U)\cr
&+&\int_{O(n)}\frac{1}{(G_n(UK))^N}\sum_{x_1,\dots,x_N\in UK\cap\Z^n}\max_{1\leq i\leq N}|\langle x_i,e_n\rangle|\chi_{B}(x_1,\dots,x_N,U)d\nu(U).
\end{eqnarray*}
On the one hand,
\begin{eqnarray*}
&&\int_{O(n)}\frac{1}{(G_n(UK))^N}\sum_{x_1,\dots,x_N\in UK\cap\Z^n}\max_{1\leq i\leq N}|\langle x_i,e_n\rangle|\chi_{B}(x_1,\dots,x_N,U)d\nu(U)\cr
&\leq&\int_{O(n)}\frac{1}{(G_n(UK))^N}\sum_{x_1,\dots,x_N\in UK\cap\Z^n}ah_q(U)\chi_{B}(x_1,\dots,x_N,U)d\nu(U)\cr
&\leq&\int_{O(n)}\frac{1}{(G_n(UK))^N}\sum_{x_1,\dots,x_N\in UK\cap\Z^n}ah_q(U)d\nu(U)\cr
&=&a\int_{O(n)}h_q(U)d\nu(U).
\end{eqnarray*}
On the other hand, let us call, for every $k\in\N\cup\{0\}$, $A_k$ the following subset of $(\Z^n)^N\times O(n)$:
$$
A_k=\left\{(x_1,\dots,x_N,U)\in A\,:\,2^kah_q(U)\leq\max_{1\leq i\leq N}|\langle x_i,e_n\rangle|<2^{k+1}ah_q(U)\right\}
$$
Notice that if $k_1\neq k_2$ then $A_{k_1}\cap A_{k_2}=\emptyset$ and $\displaystyle{A=\bigcup_{k=0}^\infty A_k}$. Then,
\begin{eqnarray*}
&&\int_{O(n)}\frac{1}{(G_n(UK))^N}\sum_{x_1,\dots,x_N\in UK\cap\Z^n}\max_{1\leq i\leq N}|\langle x_i,e_n\rangle|\chi_{A}(x_1,\dots,x_N,U)d\nu(U)\cr
&=&\int_{O(n)}\frac{1}{(G_n(UK))^N}\sum_{x_1,\dots,x_N\in UK\cap\Z^n}\max_{1\leq i\leq N}|\langle x_i,e_n\rangle|\sum_{k=0}^\infty\chi_{A_k}(x_1,\dots,x_N,U)d\nu(U)\cr
&\leq&\int_{O(n)}\frac{1}{(G_n(UK))^N}\sum_{x_1,\dots,x_N\in UK\cap\Z^n}\sum_{k=0}^\infty 2^{k+1}ah_q(U)\chi_{A_k}(x_1,\dots,x_N,U)d\nu(U)\cr
&=&2a\int_{O(n)}h_q(U)\sum_{k=0}^\infty 2^k\frac{1}{(G_n(UK))^N}\sum_{x_1,\dots,x_N\in UK\cap\Z^n}\chi_{A_k}(x_1,\dots,x_N,U)d\nu(U).\cr
\end{eqnarray*}
Applying, for any $U\in O(n)$, Lemma \ref{lem:ProbabilisticEstimateL} with $L=UK$ we have that if $(X_U)_1,\dots (X_U)_N$ are independent random vectors uniformly distributed on $UK\cap\Z^n$, then
\begin{eqnarray*}
&&2a\int_{O(n)}h_q(U)\sum_{k=0}^\infty 2^k\frac{1}{(G_n(UK))^N}\sum_{x_1,\dots,x_N\in UK\cap\Z^n}\chi_{A_k}(x_1,\dots,x_N,U)d\nu(U)\cr
&=&2a\int_{O(n)}h_q(U)\sum_{k=0}^\infty 2^k\Pro\left(2^{k+1}ah_q(U)>\max_{1\leq i\leq N}|\langle (X_U)_i,e_n\rangle|\geq 2^k ah_q(U)\right)d\nu(U)\cr
&\leq&2a\int_{O(n)}h_q(U)\sum_{k=0}^\infty 2^k\Pro\left(\max_{1\leq i\leq N}|\langle (X_U)_i,e_n\rangle|\geq 2^k ah_q(U)\right)d\nu(U)\cr
&\leq&2a\int_{O(n)}h_q(U)\sum_{k=0}^\infty 2^k N(2^ka)^{-q}d\nu(U)\cr
&=&2Na^{1-q}\int_{O(n)}h_q(U)\sum_{k=0}^\infty \frac{1}{2^{(q-1)k}}d\nu(U)\cr
&=&\frac{2Na^{1-q}}{1-\frac{1}{2^{q-1}}}\int_{O(n)} h_q(U)d\nu(U).
\end{eqnarray*}
Taking $a=e$ and $q=\log N\geq\log 3>1$ we have that $\displaystyle{\frac{2}{1-\frac{1}{2^{q-1}}}\leq \frac{2}{1-\frac{1}{2^{\log 3-1}}}=C_1}$ and
\begin{eqnarray*}
&&\int_{O(n)}\frac{1}{(G_n(UK))^N}\sum_{x_1,\dots,x_N\in UK\cap\Z^n}\max_{1\leq i\leq N}|\langle x_i,e_n\rangle|d\nu(U)\cr
&\leq& (e+eC_1)\int_{O(n)} h_{\log N}(U)d\nu(U)\\
&=&(e+eC_1)\int_{O(n)}\left(\frac{1}{G_n(UK)}\sum_{x\in UK\cap\Z^n}|\langle x,e_n\rangle|^{\log N}\right)^\frac{1}{\log N}d\nu(U),\cr
\end{eqnarray*}
which proves Theorem \ref{thm:DiscreteUpperBoundMeanWidth} with $C=e+eC_1$.
\end{proof}

As a consequence, we obtain the right-hand side inequality in \eqref{eq:MeanWidthRotations} and, equivalently, the right-hand side inequality in \eqref{eq:MeanWidthContinuous}.

\begin{cor}\label{cor:ContinuousUpperBoundMeanWidth}
There exists an absolute constant $C>0$ such that for every convex body $K\subseteq\R^n$ with $0\in K$ and every $N\in\N$ with $N\geq 3$,
\begin{eqnarray*}
&&\int_{O(n)}\frac{1}{|UK|^N}\int_{UK}\dots\int_{UK}\max_{1\leq i\leq N}|\langle x_i,e_n\rangle|dx_N\dots dx_1d\nu(U)\cr
&\leq& C\int_{O(n)}\left(\frac{1}{|UK|}\int_{UK}|\langle x,e_n\rangle|^{\log N}dx\right)^\frac{1}{\log N}d\nu(U).
\end{eqnarray*}
\end{cor}

\begin{proof}
By Theorem \ref{thm:DiscreteUpperBoundMeanWidth} we have that there exists $C>0$ such that for any $\lambda>0$, any convex body $K\subseteq\R^n$ with $0\in K$ and every $N\in\N$ with $N\geq 3$ we have that
\begin{eqnarray*}
&&\int_{O(n)}\frac{1}{(G_n(U(\lambda K)))^N}\sum_{x_1,\dots,x_N\in U(\lambda K)\cap\Z^n}\max_{1\leq i\leq N}|\langle x_i,e_n\rangle|d\nu(U)\cr
&\leq& C\int_{O(n)}\left(\frac{1}{G_n(U(\lambda K))}\sum_{x\in U(\lambda K)\cap\Z^n}|\langle x,e_n\rangle|^{\log N}\right)^\frac{1}{\log N}d\nu(U).
\end{eqnarray*}
Equivalently,
\begin{eqnarray*}
&&\int_{O(n)}\frac{1}{(G_n(\lambda UK))^N}\sum_{x_1,\dots,x_N\in \lambda UK\cap\Z^n}\max_{1\leq i\leq N}\left|\left\langle \frac{x_i}{\lambda},e_n\right\rangle\right|d\nu(U)\cr
&\leq& C\int_{O(n)}\left(\frac{1}{G_n(\lambda UK)}\sum_{x\in \lambda UK\cap\Z^n}\left|\left\langle \frac{x}{\lambda},e_n\right\rangle\right|^{\log N}\right)^\frac{1}{\log N}d\nu(U).
\end{eqnarray*}
 Lemma \ref{lemma:limite} with $m=N$ and $q=1$ yields
\begin{eqnarray*}
	&&\lim_{\lambda\to\infty}\int_{O(n)}\frac{1}{(G_n(\lambda UK))^N}\sum_{x_1,\dots,x_N\in \lambda UK\cap\left(\Z^n\right)}\max_{1\leq i\leq N}\left|\left\langle \frac{x_i}{\lambda},e_n\right\rangle\right|d\nu(U)\cr
	&=&\int_{O(n)}\frac{1}{|UK|^N}\int_{UK}\dots\int_{UK}\max_{1\leq i\leq N}|\langle x_i,e_n\rangle|dx_N\dots dx_1d\nu(U).
\end{eqnarray*}
Using now Lemma \ref{lemma:limite} with $m=1$ and $q=\log N$ we get
\begin{eqnarray*}
	&&\lim_{\lambda\to\infty}\int_{O(n)}\left(\frac{1}{G_n(\lambda UK)}\sum_{x\in \lambda UK\cap\Z^n}\left|\left\langle \frac{x}{\lambda},e_n\right\rangle\right|^{\log N}\right)^\frac{1}{\log N}d\nu(U)\cr
	&=&\int_{O(n)}\left(\frac{1}{|UK|}\int_{UK}|\langle x, e_n\rangle|^{\log N}dx\right)^\frac{1}{\log N}d\nu(U).\
\end{eqnarray*}
and the result follows.
\end{proof}

Let us now prove Theorem \ref{thm:DiscreteLowerBoundMeanWidth}.

\begin{proof}[Proof of Theorem \ref{thm:DiscreteLowerBoundMeanWidth}]Let $C>0$ be twice the absolute constant given by Theorem~\ref{thm:DiscreteBorell}
 and let $X$ be a random vector uniformly distributed on $K\cap\Z^n$. Then
$$
\frac{\left(\E|\langle X,e_n\rangle|^{2p}\right)^\frac{1}{2p}}{(\E|\langle X,e_n\rangle|^p)^\frac{1}{p}}\leq CD_0(K,p).
$$
We can assume, without loss of generality, that $C\geq e$. By the Paley-Zygmund inequality, for any $1\leq p\leq q$,
\begin{eqnarray*}
\Pro\left(|\langle X,e_n\rangle|\geq \frac{1}{2}\left(\E|\langle X,e_n\rangle|^p\right)^\frac{1}{p}\right)&=&\Pro\left(|\langle X,e_n\rangle|^p\geq \frac{1}{2^p}\E|\langle X,e_n\rangle|^p\right)\cr
&\geq&\left(1-\frac{1}{2^p}\right)^2\frac{(\E|\langle X,e_n\rangle|^p)^2}{\E|\langle X,e_n\rangle|^{2p}}\cr
&\geq&\frac{1}{4(CD_0(K,p))^{2p}}\geq\frac{1}{4(CD_{q}(K))^{2p}}.
\end{eqnarray*}

Choosing $1\leq p_1=\frac{\log N}{2\log (CD_{q}(K))}\leq\frac{\log N}{2\log C}\leq\log N\leq q$ we obtain
$$
\Pro\left(|\langle X,e_n\rangle|\geq \frac{1}{2}\left(\E|\langle X,e_n\rangle|^{p_1}\right)^\frac{1}{p_1}\right)\geq\frac{1}{4N}.
$$
Therefore, if $X_1,\dots, X_N$ are independent random vectors uniformly distributed on $K\cap\Z^n$, we have that
\begin{eqnarray*}
&&\Pro\left(\max_{1\leq i\leq N}|\langle X_i,e_n\rangle|\geq \frac{1}{2}\left(\E|\langle X,e_n\rangle|^{p_1}\right)^\frac{1}{p_1}\right)\cr
&=&1-\Pro\left(\max_{1\leq i\leq N}|\langle X_i,e_n\rangle|<\frac{1}{2}\left(\E|\langle X,e_n\rangle|^{q}\right)^\frac{1}{p_1}\right)\cr
&=&1-\Pro\left(|\langle X_i,e_n\rangle| <\frac{1}{2}\left(\E|\langle X,e_n\rangle|^{p_1}\right)^\frac{1}{p_1}\right)^N\cr
&=&1-\left(1-\Pro\left(|\langle X_i,e_n\rangle|\geq\frac{1}{2}\left(\E|\langle X,e_n\rangle|^{p_1}\right)^\frac{1}{p_1}\right)\right)^N\cr
&\geq&1-\left(1-\frac{1}{4N}\right)^N=1-e^{N\log\left(1-\frac{1}{4N}\right)}\geq1-e^{-\frac{1}{4}}.
\end{eqnarray*}
Therefore, by Markov's inequality,
\begin{eqnarray*}
\E\max_{1\leq i\leq N}|\langle X_i,e_n\rangle|&\geq&\frac{1}{2}\left(\E|\langle X,e_n\rangle|^{p_1}\right)^\frac{1}{p_1}\cdot \Pro\left(\max_{1\leq i\leq N}|\langle X_i,e_n\rangle|\geq \frac{1}{2}\left(\E|\langle X,e_n\rangle|^{p_1}\right)^\frac{1}{p_1}\right)\cr
&\geq&\frac{1-e^{-\frac{1}{4}}}{2}\left(\E|\langle X,e_n\rangle|^{p_1}\right)^\frac{1}{p_1}.
\end{eqnarray*}
This proves Theorem \ref{thm:DiscreteLowerBoundMeanWidth} with $c=\frac{1-e^{-\frac{1}{4}}}{2}$.
\end{proof}

Let us now see that Corollary \ref{cor:DiscreteLowerBoundMeanWidth} implies the left-hand side inequality in \eqref{eq:MeanWidthRotations}.

\begin{cor}
There exist absolute constants $c_1, c_1^\prime, c>0$ such that if $K\subseteq\R^n$ is a convex body with $0\in\operatorname{int}K$ and $N\in\N$ satisfies that $N\geq c$ then
\begin{eqnarray*}
&&\int_{O(n)}\frac{1}{|UK|^N}\int_{UK}\dots\int_{UK}\max_{1\leq i\leq N}|\langle x_i,e_n\rangle|dx_N\dots dx_1d\nu(U)\cr
&\geq& c_1\int_{O(n)}\left(\frac{1}{|UK|}\int_{UK}|\langle x,e_n\rangle|^{c_1^\prime\log N}dx\right)^\frac{1}{c_1^\prime\log N}d\nu(U).
\end{eqnarray*}
\end{cor}

\begin{proof}
Let $c=4C^2$, where $C$ is the constant given by Theorem \ref{thm:DiscreteLowerBoundMeanWidth} and Corollary \ref{cor:DiscreteLowerBoundMeanWidth}. Let $K\subseteq\R^n$ be a convex body with $0\in\operatorname{int}K$ and let $N\geq c=4C^2$. Let us take $q=\log N$. Since $\displaystyle{\lim_{\lambda\to\infty}D_{q}(\lambda K)=1}$, there exists $\lambda_0>0$ such that for every $\lambda\geq\lambda_0$ we have $B_2^n\subseteq \lambda K$ and   $D_{q}(\lambda K)\leq 2$. Thus
$$
(CD_{q}(\lambda K))^2\leq 4C^2\leq N= e^{q}.
$$
Therefore, by Corollary \ref{cor:DiscreteLowerBoundMeanWidth}, there exists an absolute constant $c_1>0$ such that we have that for every $\lambda\geq\lambda_0$,
\begin{eqnarray*}
&&\int_{O(n)}\frac{1}{(G_n(\lambda UK))^N}\sum_{x_1,\dots,x_N\in \lambda UK\cap\Z^n}\max_{1\leq i\leq N}|\langle x_i,e_n\rangle|d\nu(U)\cr
&\geq& c_1\int_{O(n)}\left(\frac{1}{G_n(\lambda UK)}\sum_{x\in\lambda UK}|\langle x,e_n\rangle|^\frac{\log N}{2\log(CD_{q}(K))}dx\right)^\frac{2\log(CD_{q}(K))}{\log N}d\nu(U)\cr
&\geq& c_1\int_{O(n)}\left(\frac{1}{G_n(\lambda UK)}\sum_{x\in\lambda UK}|\langle x,e_n\rangle|^\frac{\log N}{2\log(2C)}\right)^\frac{2\log(2C)}{\log N}d\nu(U)\cr
&=& c_1\int_{O(n)}\left(\frac{1}{G_n(\lambda UK)}\sum_{x\in\lambda UK}|\langle x,e_n\rangle|^{c_1^\prime\log N}\right)^\frac{1}{c_1^\prime\log N}d\nu(U),\cr
\end{eqnarray*}
with $c_1\prime=\frac{1}{2\log(2C)}$. Equivalently,
\begin{eqnarray*}
&&\int_{O(n)}\frac{1}{(G_n(\lambda UK))^N}\sum_{x_1,\dots,x_N\in \lambda UK\cap\Z^n}\max_{1\leq i\leq N}\left|\left\langle \frac{x_i}{\lambda},e_n\right\rangle\right|d\nu(U)\cr
&=& c_1\int_{O(n)}\left(\frac{1}{G_n(\lambda UK)}\sum_{x\in\lambda UK\cap\Z^n}\left|\left\langle \frac{x}{\lambda},e_n\right\rangle\right|^{c_1^\prime\log N}\right)^\frac{1}{c_1^\prime\log N}d\nu(U).\cr
\end{eqnarray*}
Taking limit as $\lambda\to\infty$ and using Lemma \ref{lemma:limite} with $m=N$ and $q=1$ and with $m=1$ and $q=c^\prime\log N$, we obtain
\begin{eqnarray*}
&&\int_{O(n)}\frac{1}{|UK|^N}\int_{UK}\dots\int_{UK}\max_{1\leq i\leq N}|\langle x_i,e_n\rangle|dx_N\dots dx_1d\nu(U)\cr
&\geq& c_1\int_{O(n)}\left(\frac{1}{|UK|}\int_{UK}|\langle x,e_n\rangle|^{c_1^\prime\log N}dx\right)^\frac{1}{c_1^\prime\log N}d\nu(U).
\end{eqnarray*}
\end{proof}

\subsection*{Acknowledgments} We are very grateful to the anonymous referees for their careful reading of the paper and valuable comments, which were very useful to improve the paper.



\begin{thebibliography}{99}

\bibitem{A} {\sc F. Affentranger.}
\textit{The convex hull of random points with spherically symmetric distributions.}
Rend. Semin. Mat. Univ. Politec. Torino {\bf 49} (1991), 359--383.

\bibitem{AS}{\sc F. Affentranger, R. Schneider.}
\textit{Random projections of regular simplices.}
 Discrete Comput. Geom. {\bf 7} (1992), 219--226.

\bibitem{AHZ}{\sc M. Alexander, M. Henk, A. Zvavitch.}
\textit{A discrete version of Koldobsky's slicing inequality.}
Israel J. Math. {\bf 222} (1) (2017) 261--278.


\bibitem{ALY}{\sc D.~Alonso-Guti\'errez, E.~Lucas, J.~Yepes Nicol\'as.}
\textit{On Rogers-Shephard type inequalities for the lattice point enumerator.}
Commun. Contemp. Math. {\bf 25} (8) 2250022, (2023).

\bibitem{AP}{\sc D.~Alonso-Guti\'errez, J. Prochno.}
\textit{On the geometry of random convex sets between polytopes and zonotopes.}
J. Math. Anal. Appl. {\bf 450} (2017), 670--690.

\bibitem{B}{\sc I. B\'ar\'any.}
\textit{ Random polytopes, convex bodies, and approximation.}
Stochastic Geometry, in: Lecture Notes in Math {\bf 1892}, Springer, Berlin, (2007), 77--118.


\bibitem{BV}{\sc I. B\'ar\'any, V. Vu.}
\textit{Central limit theorems for Gaussian polytopes.}
Ann. Probab. {\bf 35} (2007), 1593--1621.

\bibitem{BGVV} {\sc S. Brazitikos, A. Giannopoulos, P. Valettas, B.~H. Vritsiou.} \textit{Geometry of Isotropic Convex Bodies}.
Mathematical Surveys and Monographs \textbf{196} (American Mathematical Society, Providence, RI., 2014).

\bibitem{CS}{\sc L. J. Corwin, R. H. Szczarba.}
\textit{Multivariable calculus.}
Monographs and textbooks in pure and applied matematics {\bf 64} (Marcel Dekker, 1979).

\bibitem{DGT}{\sc N. Dafnis, A. Giannopoulos, A. Tsolomitis.}
\textit{Quermassintegrals and asymptotic shape of random polytopes in an isotropic convex body.}
Michigan Math. J. {\bf 62} (2013), 59--79.

\bibitem{DT}{\sc D.L. Donoho, J. Tanner.}
\textit{Counting faces of randomly projected polytopes when the projection radically lowers dimension.}
 J. Amer. Math. Soc. {\bf 22} (1) (2009), 1--53.

\bibitem{FGP}{\sc B.~Fleury, O.~Gu\'edon, G.~Paouris}, \textit{A stability result for mean width of $L_p$-centroid bodies},
Adv. Math. {\bf 214} (2007), 865--877.

\bibitem{FH}{\sc A. Freyer, M. Henk.}
\textit{Bounds on the Lattice Point Enumerator via Slices and Projections.}
Discrete Comput. Geom. {\bf 67} (2022), 895--918.

\bibitem{GG}{\sc R. J. Gardner, P. Gronchi}
\textit{A Brunn-Minkowski inequality for the integer lattice.}
Trans. Amer. Math. Soc. {\bf 353} (10) (2001), 3995--4024.

\bibitem{GSTV}{\sc A.~Giannopoulos, P.~Stavrakakis, A.~Tsolomitis, B.~H.~Vritsiou},
\textit{Geometry of the
$L_q$-centroid bodies of an isotropic log-concave measure.} Tran. Amer. Math. Soc. \textbf{367} (2015), 4569--4593.


\bibitem{G} {\sc E. D. Gluskin.}
\textit{The diameter of Minkowski compactum roughly equals to $n$}
Funktsional. Anal. i Prilozhen. {\bf 15} (1) (1981), 72--73 . English Trans: Funct. Anal. Appl. {\bf 15} (1) (1981),  57--58.

\bibitem{GL} {\sc E. D. Gluskin, A.E. Litvak}
\textit{A remark on vertex index of the convex bodies.}
GAFA. Lecture Notes in Math. {\bf 2050} (2012), 255--265. Springer, Berlin.

\bibitem{GT}{\sc J. Grote, C. Th\"ale.}
\textit{Gaussian polytopes: A cumulant-based approach.}
J. Complexity {\bf 47} (2018), 1--41.

\bibitem{HR}{\sc D. Hug, M. Reitzner.}
\textit{Gaussian polytopes: variances and limit theorems.}
Adv. Appl. Probab. {\bf 37} (2005), 297--320.

\bibitem{ILY}{\sc D. Iglesias, E. Lucas, J. Yepes Nicol\'as.}
\textit{On discrete Brunn-Minkowski and isoperimetric type inequalities.}
Discrete Math. {\bf 345} (1) (2022), 112640.

\bibitem{IYNZ}{\sc D.~Iglesias, J.~Yepes Nicol\'as, A.~Zvavitch.}
\textit{Brunn-Minkowski type inequalities for the lattice point enumerator.}
Advances in Mathematics {\bf 370} (2020), 107193.

\bibitem{LPRT}{\sc A.E. Litvak, A. Pajor, M. Rudelson, N. Tomczak-Jaegermann.}
\textit{Smallest singular value of random matrices and geometry of random polytopes.}
 Adv. Math. {\bf 195} (2) (2005), 491--523.

\bibitem{LRT}{\sc A.E. Litvak, M. Rudelson, N. Tomczak-Jaegermann.}
\textit{On approximation by projections of polytopes with few facets.}
Israel J. Math. {\bf 203} (2014), 141--160.

\bibitem{LYZ}{\sc E.~Lutwak, D.~Yang and G.~Zhang}. \textit{$L^p$
affine isoperimetric inequalities}, J. Differential
Geom. {\bf 56} (2000), 111--132.

\bibitem{MT}{\sc P. Mankiewicz, N. Tomczak-Jaegermann}
\textit{Quotients of finite-dimensional Banach spaces; random phenomena.}
Johnson,W.B., Lindenstrauss, J. (eds.) Handbook in the Geometry of Banach Spaces, vol 2, (2001)  1201--1246. Elsevier, Amsterdam.

\bibitem{MS}{\sc V.~D.~Milman, G.~Schechtman}. \textit{Asymptotic theory of finite dimensional normed spaces}, Lecture Notes in Mathematics \textbf{1200} (1986), Springer, Berlin


\bibitem{P}{\sc G. Paouris}, Concentration of mass on convex bodies, Geom. Funct. Anal. {\bf 16} (2006), 1021--1049.

\bibitem{PW}{\sc G. Paouris, E. Werner}.
\textit{Relative entropy of cone measures and $\ell_p$ centroid bodies.}
Proc. London Math. Soc. {\bf 104} (2), (2012), 253--286.

\bibitem{RS}{\sc A. R\'enyi, R. Sulanke.}
\textit{\"Uber die konvexe H\"ulle von nzuf\"allig gew\"ahlten Punkten.}
Z. Wahrsch. Verw. Gebiete 2 (1963), 75--84.

\bibitem{RS2}{\sc A. R\'enyi, R. Sulanke.}
\textit{\"Uber die konvexe H\"ulle von nzuf\"allig gew\"ahlten Punkten II.}
Z. Wahrsch. Verw. Gebiete 3 (1964), 138--147.

\bibitem{RS3}{\sc A. R\'enyi, R. Sulanke.}
\textit{Zuf\"allige konvexe Polygone in einem Ringgebiet.}
Z. Wahrsch. Verw. Gebiete 9 (1968), 146--157.

\bibitem{Sch} {\sc R.~Schneider.}
\textit{Convex Bodies: The Brunn--Minkowski Theory,  2nd expanded edition}.
Encyclopedia Math. Appl., vol. 151, Cambridge University Press, Cambridge, 2014.

\bibitem{S}{\sc J. J. Sylvester.}
Question 1491. The Educational Times. 1864. London.

\bibitem{TV}{\sc T.~Tao and V.~H.~Vu}. \textit{Additive Combinatorics}. Cambridge Studies in Advanced Mathematics, 105. Cambridge University Press, Cambridge, 2006.
\end{thebibliography}
\end{document}